\documentclass[a4paper, 12pt, oneside]{amsart}
\usepackage{amsrefs}
\usepackage{bm}

\title[Capturing the APS index from the lattice]{Capturing the Atiyah--Patodi--Singer index from the lattice}

\author[S. Aoki]{Shoto Aoki}
\author[H. Fujita]{Hajime Fujita}
\author[H. Fukaya]{Hidenori Fukaya}
\author[M. Furuta]{Mikio Furuta}
\author[S. Matsuo]{Shinichiroh Matsuo}
\author[T. Onogi]{Tetsuya Onogi}
\author[S. Yamaguchi]{Satoshi Yamaguchi}

\address[SA]{%
  Interdisciplinary Theoretical and Mathematical Sciences Program (iTHEMS) \\
  RIKEN \\
  Wako, Japan
}
\email{shotoaoki@g.ecc.u-tokyo.ac.jp}

\address[HFuj]{
Faculty of Science, Japan Women’s University, Mejirodai, Bunkyo-ku, Tokyo 112-8681, Japan
}
\email{fujitah@fc.jwu.ac.jp}

\address[HFuk, TO, and SY]{%
  Department of Physics \\
  Osaka University \\
  Osaka, Japan
}
\email{hfukaya@het.phys.sci.osaka-u.ac.jp}
\email{onogi@phys.sci.osaka-u.ac.jp}
\email{yamaguch@het.phys.sci.osaka-u.ac.jp}

\address[MF]{%
  Graduate School of Mathematical Sciences \\
  The University of Tokyo \\
  Tokyo, Japan
}
\email{furuta@ms.u-tokyo.ac.jp}

\address[SM]{%
  Graduate School of Mathematics \\
  Nagoya University \\
  Nagoya, Japan
}
\email{shinichiroh@math.nagoya-u.ac.jp}

\thanks{Preprint number: OU-HET-1300}

\usepackage[a4paper]{geometry}
\usepackage{microtype}

\usepackage[utf8]{inputenc}
\usepackage[T1]{fontenc}

\usepackage{amsmath}
\usepackage{stmaryrd}  
\usepackage{amscd}
\usepackage{amssymb}
\usepackage{mathtools}
\mathtoolsset{showonlyrefs=true}

\theoremstyle{plain}
\newtheorem{theorem}{Theorem}
\newtheorem{proposition}[theorem]{Proposition}
\newtheorem{lemma}[theorem]{Lemma}
\newtheorem{corollary}[theorem]{Corollary}

\theoremstyle{definition}
\newtheorem{definition}[theorem]{Definition}

\theoremstyle{remark}
\newtheorem{remark}[theorem]{Remark}

\usepackage{hyperref}
\hypersetup{bookmarksnumbered}
\hypersetup{hidelinks}

\newcommand{\R}{\mathbb{R}}
\newcommand{\Z}{\mathbb{Z}}

\newcommand{\cliffordalgebra}{\mathrm{Cl}}
\DeclareMathOperator{\cliffordmultiplication}{\sigma}

\newcommand{\T}{\mathbb{T}}

\DeclareMathOperator{\End}{End}

\newcommand{\lattice}[2][]{\widehat{#2}}

\newcommand{\lT}{\lattice{\T^d}}
\newcommand{\lX}{\lattice{X}}
\newcommand{\lE}{\lattice{E}}
\newcommand{\lkappa}{\lattice{\kappa}}

\DeclareMathOperator{\paralleltransport}{U}

\DeclareMathOperator{\differenceoperator}{\lattice{\nabla}}
\newcommand{\forwarddifferenceoperator}{\differenceoperator^{\mathrm{f}}}
\newcommand{\backwarddifferenceoperator}{\differenceoperator^{\mathrm{b}}}
\DeclareMathOperator{\naiveDiracoperator}{\lattice{D}^{\mathrm{naive}}}

\DeclareMathOperator{\Wilsonterm}{W}

\DeclareMathOperator{\WilsonDiracoperator}{\lattice{D}^{\mathrm{wilson}}}

\DeclareMathOperator{\interpolatorOperator}{\iota}
\newcommand{\interpolator}[1]{\interpolatorOperator_{#1}}

\newcommand{\C}{\mathbb{C}}
\DeclareMathOperator{\id}{id}

\DeclareMathOperator{\APSindex}{Ind_{APS}}
\DeclareMathOperator{\gAPSindex}{Ind_{gAPS}}

\DeclareMathOperator{\Ker}{Ker}
\newcommand{\DiracOperatorOnY}{{D_Y}}
\newcommand{\restrictedTo}[2]{\left. #1 \right|_{#2}}
\newcommand{\DiracOperatorOnCylinderAPS}{\widehat{\mathcal{D}}}
\newcommand{\HH}{{\mathcal H}}

\DeclarePairedDelimiter{\abs}{\lvert}{\rvert}

\DeclarePairedDelimiterX\innerproduct[2]{\langle}{\rangle}{#1,#2}
\DeclarePairedDelimiterX\anticommutator[2]{\lbrace}{\rbrace}{#1,#2}

\DeclareMathOperator{\Image}{Im}
\DeclareMathOperator{\Kernel}{Ker}

\DeclareMathOperator{\cliffordGamma}{\Gamma}
\newcommand{\starHom}{\mathrm{Hom}^*}

\newcommand{\point}{\mathit{pt}}

\newcommand{\Kcocycle}{\mathcal{K}}

\newcommand{\bddorRiesz}{\bullet}
\newcommand{\bounded}{\mathrm{bounded}}
\newcommand{\Riesz}{\mathrm{Riesz}}
\newcommand{\Cayley}{\mathrm{Cayley}}

\DeclareMathOperator{\bddtoRiesz}{\tau^{\bounded}_{\Riesz}}
\DeclareMathOperator{\Riesztobdd}{\tau^{\Riesz}_{\bounded}}
\DeclareMathOperator{\TRtohalf}{T_{\R \to [-1/2,1/2]}}
\DeclareMathOperator{\Tunittohalf}{T_{[-1,1] \to [-1/2,1/2]}}

\newcommand{\boundedoperators}{\mathcal{B}}
\DeclarePairedDelimiter{\operatornorm}{\lVert}{\rVert_{\mathrm{op}}}
\newcommand{\unboundedoperators}{\mathcal{C}}
\newcommand{\boundedselfadjointoperators}{\boundedoperators_{\mathit{sa}}}
\newcommand{\unboundedselfadjointoperators}{\unboundedoperators_{\mathit{sa}}}
\newcommand{\Riesztransform}{T_{\Riesz}}
\newcommand{\unitaryoperators}{\mathcal{U}}
\newcommand{\Cayleytransform}{T_{\Cayley}}

\DeclareMathOperator{\spectrum}{\sigma}

\newcommand{\Hilbertbundles}{\mathrm{Hilbert}}
\DeclarePairedDelimiter{\cardinality}{\lvert}{\rvert}
\newcommand{\modulosettheory}[1]{\operatorname{mod} {\cong_{#1}}}

\newcommand{\matrixalgebra}{\mathrm{Mat}}

\usepackage{color}

\begin{document}

\begin{abstract}
    We construct a formulation of the Atiyah-Patodi-Singer index
    of Dirac operators in lattice gauge theory
    for domains with compact boundaries in a flat torus.
    The key idea is to exploit its equality to the spectral flow of the
    domain-wall fermion Dirac operators, which we generalize in this work
    to cases without product structure near the boundary.
    We prove that, for sufficiently small lattice spacings,
this formulation correctly captures the continuum Atiyah–Patodi–Singer index.
\\
\end{abstract}

\maketitle

\tableofcontents

\section{Introduction}

Lattice gauge theory offers a powerful tool in particle physics
to compute quantum field theories from the first principle.
By coarse-graining the space time into a discrete lattice space,
the functional integral becomes mathematically well-defined
and numerically calculable.
However, a naive discretization often fails to maintain
crucial properties of the target continuum theory.
Topology is one of them.

In particular, the Fredholm index of Dirac operators\footnote{
To be precise, we consider Dirac-type operators, 
which form a broader class than the Dirac operators 
associated with spin structures or spin$^c$ structures.
In our work, we simply refer to Dirac-type operators as Dirac operators.
} 
(we refer to it the Atiyah-Singer(AS) index for comparison with
the Atiyah-Patodi-Singer index)
has been a main challenging subject to formulate in lattice gauge theory.
The most successful approach was given by the overlap Dirac operator
on even-dimensional square lattices \cite{Neuberger:1997fp} (and a similar formulation by
\cite{Hasenfratz:1998ri}),
which realizes the $\mathbb{Z}_2$-grading operator,
or chirality operator in physics \cite{Luscher:1998pqa},
with a slight modification to satisfy the so-called Ginsparg-Wilson
relation \cite{Ginsparg:1981bj}. 
In this approach, the index was defined
by the kernel of the overlap Dirac operator,
which is consistent with the AS index of
the continuum Dirac operator on a torus. 
See \cite{Adams:1998eg} for a mathematical justification
and \cite{Clancy:2023ino} for a general extension of this formulation.

In our previous work \cite{Aoki:2024sjc},
we addressed a different approach on a lattice
to extracting the AS index from massive Dirac operators.
We used a mathematical relation between the AS index of the massless
($\mathbb{Z}_2$-graded) Dirac operator and a family of
massive Dirac operators, which is well-known in $K$-theory as 
the suspension isomorphism $K^0(\mathrm{pt})\cong K^1(I,\partial I)$.
Here $\mathrm{pt}$ denotes a one-point set, 
and $I$ is an interval with its boundary $\partial I$, consisting of the two end points.
We gave a mathematical proof that
at sufficiently small lattice spacings, we can construct a family of 
the massive Wilson Dirac operators on a lattice, 
which can be identified as an element of $K^1(I,\partial I)$ and
its spectral flow\footnote{
The spectral flow approach was  empirically known \cite{Itoh:1987iy}
(even before the overlap fermion was known)
to reproduce the index but the mathematical background
was not rigorously discussed.
} gives the same value of
the AS index in the continuum theory. 
Since the $\mathbb{Z}_2$-grading structure
is lost in the massive Dirac operator from the beginning, 
the Ginsparg-Wilson relation is not required in our formulation
and the standard Wilson Dirac operator is good enough.
We note that there are also different mathematical formulations
\cite{MR4275791,MR4407739} using the Wilson Dirac operator.

In this work, we attempt an important generalization of the  index, the Atiyah--Patodi--Singer (APS) index~\cite{MR397797, Atiyah:1976jg, Atiyah:1976qjr} of Dirac operators on manifolds with boundary.  
The APS index is studied in physics to understanding 
the bulk-boundary correspondence of the fermion anomaly in the symmetry-protected topological phases 
\cite{Witten:2015aba,Yu:2016hcq,Witten:2019bou,Ivanov:2020fsz,Kobayashi:2021jbn,Kanno:2021bze,Onogi:2021slv,Pedersen:2023fpd,Nguyen:2024wck}.
Compared to the AS index on closed manifolds, the APS index is more challenging to realize on the lattice due to the following difficulties:
\begin{itemize}
  \item 
 The APS boundary condition is global and non-local, making its lattice formulation difficult.  In particular, it is not known how to impose the APS boundary condition on lattice Dirac operators.
  \item 
 The APS index is not topological, {\it i.e.}, it depends on the metric and connections of the bundle  near the boundary.  This is in contrast to the AS index, which is a topological invariant.  Therefore, one needs to control the metric and connection dependence. 
\end{itemize}

  The key idea is to exploit a mathematical relation
   between the
  APS index and the spectral flow of the domain-wall fermion Dirac operators \cite{Jackiw:1975fn, Callan:1984sa}
  shown in continuum theory \cite{Fukaya:2017tsq, Fukaya:2019qlf, Fukaya:2020tjk, Fukaya:2021sea}.
  It was originally proved in the case where the neighborhood of the boundary
  has a product metric.
  In this work, we generalize the relation to cases without product structure
  near the boundary, where the index is defined with
  the ``canonical'' boundary operator \cite{Gerd01011992, GILKEY1993129,braverman2019equivariant}.

By the domain-wall fermion formulation, we can overcome the above two difficulties as below.
\begin{itemize}
\item  We can avoid the APS boundary condition by gluing the boundary of
the original manifold we denote by $X_+$ to another manifold $X_-$ sharing the same boundary
to form a closed manifold.
Instead of the boundary condition, we assign opposite signs to the mass term on $X_+$ and $X_-$ so that 
the nontrivial geometrical information is obtained from the $X_+$ subspace only, 
which is proved to be the APS index on $X_+$.
This observation indicates that the APS index on a lattice may be defined by
the spectral flow of the lattice domain-wall fermion Dirac operators \cite{Kaplan:1992bt,Shamir:1993zy}.
Once formulated in terms of the spectral flow, the proof of \cite{Aoki:2024sjc} can be applied almost straightforwardly.
\item We can also avoid the problem that the APS index is not a topological invariant,
by assuming that the target domain-wall fermion Dirac operator is invertible.
With this assumption, we can construct the lattice domain-wall fermion Dirac operator family
as a mathematically well-defined element of $K^1(I,\partial I)$.
\end{itemize}





Here is a summary of this paper.
We describe the relation between the APS index and the spectral flow of the
domain-wall fermion Dirac operators based on
\cite{Fukaya:2017tsq, Fukaya:2019qlf, Fukaya:2020tjk, Fukaya:2021sea}.
In this work, we further generalize
this relation to cases without product structure near the boundary (Theorem~\ref{thm:FFMOYYg}).
We give a formulation of $K$ and $KO$ groups
in such a way that continuum unbounded Dirac operators
and finite lattice bounded Dirac operators
are simultaneously handled.
Although such a formulation was essentially known in the literature
\cite{MR2114489} and discussed in our previous paper \cite{Aoki:2024sjc}
as well, we would like to present it in a selfcontained and comprehensive way
treating the $K$ and $KO$ groups with arbitrary degree
on general pairs of base spaces.
Then we consider a direct sum of continuum and lattice domain-wall fermion
Dirac operators and prove in our main theorem (Theorem~\ref{goal})
that the combined Dirac operator gives a trivial element of $K^1(I,\partial I)$.
In the proof, our finite element interpolator 
between the functions on the lattice and those in the continuum space,
as well as its action on the domain-wall mass term,
plays a key role to ensure that the combined operator is invertible.
The equality between the two spectral flows
immediately follows (Theorem~\ref{goalgoal}).
As far as we know, this is the first mathematically rigorous
formulation of the APS index on a lattice.
This formulation is sufficiently robust to allow extensions to
the systems with additional symmetries.
We explicitly describe an application to the mod-two version
of the APS index when the lattice domain-wall fermion Dirac operator is real (Theorems~\ref{goal-mod2} and \ref{goalgoal-mod2}).
We note that the physics part of this work was
already published in \cite{Aoki:2025gca} 
(see also \cite{Fukaya:2019myi} which gave a perturbative discussion of the spectral flow of the domain-wall fermion Dirac operators), 
where we presented the key formulas with a summary of mathematical descriptions and 
numerical examinations of them on two-dimensional lattices. 

In this work, the base manifold $X=X_-\cup X_+$
is limited to a flat torus.
Accordingly we employ a square lattice as a discretization of $X$.
The readers may ask a question if we can extend our work
to the case where $X$ is a general curved manifold.
Currently we do not have a clear answer but
there is an interesting observation.
In our setup, the domain-wall $Y$ can be any curved submanifold of $X$.
On such a curved domain-wall on a lattice, it was shown that
a nontrivial curvature effect (gravitational background in physics) 
is induced, which is
consistent with the continuum theory \cite{Aoki:2022cwg,Aoki:2022aez,Aoki:2023lqp,Aoki:2024bwx,Clancy:2024bjb,Araki:2025xly}.
Embedding $X$ into a further higher dimensional square lattice
may be an interesting direction, which is, however, beyond the scope of this work.

The rest of the paper is organized as follows.
In Section~\ref{sec:Dirac-operators} we construct
the Wilson Dirac operator on a square lattice
from a given continuum Dirac operator on a flat torus.
In Section~\ref{sec:from-lattice-to-continuum} we define
a finite element interpolator 
between the functions on a lattice and those in continuum space and summarize
its key properties proved in our previous work \cite{Aoki:2024sjc}.
Then in Section~\ref{sec:DWandAPS}, we define the domain-wall fermion Dirac operator 
and review its relation to the APS index.
We also generalize the relation to cases without a product metric near the boundary.
In Section~\ref{sec:K} we present the formulation of $K$ and $KO$ groups
which simultaneously handles bounded and unbounded Dirac operators.
Definitions of the spectral flow and its mod-two version are given in Section~\ref{sec:sf}.
Then in Section~\ref{sec:main-theorem}, we state our main theorems
and give the proof. As a nontrivial example with symmetry,
we also present an application to the mod-two APS index of real Dirac operators in Section~\ref{sec:mod-2APS}.

\section{Wilson Dirac operators on a lattice}
\label{sec:Dirac-operators}

In this section, we will construct the Wilson Dirac operators, commonly used in lattice gauge theory, which we adopt as our discretization of the Dirac operators.

\subsection{Dirac operators in continuum space}

We set up notation.
Let $X := \T^d = (\R/\Z)^d$ be a $d$-dimensional flat torus.
Let $\bm{e}_1, \dots, \bm{e}_d$ be the standard orthonormal basis of $\R^d$.
The tangent bundle $T X$ is canonically isometric to the trivial bundle $X \times \R^d$.
The Clifford algebra $\cliffordalgebra^d$ is generated by $\{e_1, \dots, e_d\}$ subject to the anticommutation relation
\begin{equation}
\{e_i,e_j\}=  e_i e_j + e_j e_i = -2 \delta_{ij}.
\end{equation}
For simplicity, we first assume that $d$ is even.
Let $E \to X$ be a Clifford module bundle on $X$; that is, $E$ is a $\Z_2$-graded Hermitian vector bundle with a smooth map of graded algebra bundles $\cliffordmultiplication \colon X \times \cliffordalgebra^d \to \End(E)$.
We also assume that $\cliffordmultiplication(e_j)^*=-\cliffordmultiplication(e_j)$.
We denote its $\Z_2$-grading operator by $\gamma$.

Fix a Clifford connection $A$ on $E$ 
and we denote the covariant derivative in the $\bm{e}_j$ direction with respect to $A$ by $\nabla_{j}$.
Let $R_X$ be the injective radius of $X$.
For arbitrary two points $x,y \in X$ such that $|x-y|<R_X$, 
there is a unique minimal geodesic from $x$ to $y$. 
We denote the parallel transport by the connection $A$ along this minimal geodesic by
\begin{equation}
  \paralleltransport_{x,y} \colon E_{y} \to E_x,
\end{equation}
which depends smoothly on $x$ and $y$. 
Note that $\paralleltransport_{y,x}=\paralleltransport_{x,y}^{-1}$ 
and $\paralleltransport_{x,x}=\id_{E_x}$ hold\footnote{The set of the parallel transports $\{\paralleltransport_{x,y}\}$ is 
an example of the generalized link variables defined in \cite{Aoki:2024sjc}. }.

Let $\Gamma (E)$ be the space of smooth sections of $E$. 
For smooth sections $u,v \in \Gamma(E)$ and the Hermitian metric $(\cdot, \cdot)$ on $E$, we denote their inner product and the associated norm by 
\begin{equation}
 \langle u, v \rangle:= \int_X (u,v) \;d x, \quad
||v||_{L^2}:= {\langle v,  v \rangle}^{1/2},
\end{equation}
where $d x$ denotes the volume element of $X$ associated with the standard flat metric. 
Let $D : \Gamma (E)\to \Gamma (E)$ be the Dirac operator defined by 
\begin{equation}
  D u:= \sum_{j=1}^d \cliffordmultiplication(e_j) \nabla_{j} u,
\end{equation}
for $u \in \Gamma (E)$, which is a first order formally selfadjoint elliptic operator. 
We also use the $L^2_1$ norm $||  \cdot ||_{L^2_1}$ on $\Gamma(E)$ ; 
\begin{equation}
 || v||_{L^2_1} :=\Big( ||v||^2_{L^2} +  \frac{1}{m_0^2}\sum_{j=1}^d \int_X | \nabla_{j} v |^2 \;d x\Big) ^{1/2}
\end{equation}
for $v\in \Gamma(E)$, where $m_0$ is an arbitrary non-zero real number,
 which is often taken as a typical scale of the physical system we focus on.
We consider the Hilbert space $L^2(E)$ (resp. $L^2_1(E)$) of the completion of $\Gamma(E)$ by the norm $||*||_{L^2}$ (resp. $||*||_{L^2_1}$).
The Dirac operator $D$ can be extended as an unbounded selfadjoint operator (which is denoted by the same letter)
\begin{equation}
  D : L^2(E)\to L^2(E),
\end{equation}
with the domain $L_1^2(E)$.

\subsection{Wilson Dirac operators}
\label{sec:Wilson}

Let $N$ be a positive integer, and set $a := 1/N$.
Let $\lX_a := (a\Z/\Z)^d \subset X$ be the standard lattice: 
we take, for simplicity, a hypercubic lattice whose size and lattice spacing are
equal in every direction\footnote{
  It would not be difficult to consider anisotropic shape of the lattice.
  but we do not discuss that in this work.}.
Set $\lE_a := \left. E \right|_{\lX_a}$.
We simply write $\lX$ and $\lE$ when $a$ dependence is not important.
For each lattice point $z \in \lX$ and $j \in \{1,\dots,d\}$, 
we define the {\it link variables}
\begin{equation}
  \paralleltransport_{z,z+a\bm{e}_j} \colon \lE_{z+a\bm{e}_j} \to \lE_z,
\end{equation}
which is the restriction of the parallel transport $\paralleltransport_{x,y}$ onto the lattice $\lX$.
We define the forward difference operator $\forwarddifferenceoperator_j \colon \Gamma(\lE) \to \Gamma(\lE)$ by
\begin{equation}
  \big( \forwarddifferenceoperator_j u \big)(z) := \frac{\paralleltransport_{z,z+a\bm{e}_j} \big( u (z+a\bm{e}_j) \big) - u(z)}{a}
\end{equation}
for $u \in \Gamma(\lE)$ and $z \in \lX$.
We also define the backward difference operator $\backwarddifferenceoperator_j \colon \Gamma(\lE) \to \Gamma(\lE)$ by
\begin{equation}
  \big( \backwarddifferenceoperator_j u \big)(z) := \frac{u(z) - \paralleltransport_{z,z-a\bm{e}_j} \big( u (z-a\bm{e}_j) \big)}{a}
\end{equation}
for $u \in \Gamma(\lE)$ and $z \in \lX$.
Note that $\backwarddifferenceoperator_j u = - \big( \forwarddifferenceoperator_j u \big)^*$.
We then define a covariant difference operator $\differenceoperator_j \colon \Gamma(\lE) \to \Gamma(\lE)$ by $\differenceoperator_j := \big( \forwarddifferenceoperator_j + \backwarddifferenceoperator_j \big) / 2$.
Note that $\differenceoperator_j$ is skew adjoint.
Now we define the \emph{naive} Dirac operator $\naiveDiracoperator \colon \Gamma(\lE) \to \Gamma(\lE)$ by
\begin{equation}
  \naiveDiracoperator u := \sum_{j=1}^{d} \cliffordmultiplication(e_j) \differenceoperator_j u
\end{equation}
for $u \in \Gamma(\lE).$
Note that $\naiveDiracoperator$ is selfadjoint.

A fundamental lesson of lattice gauge theory is that the naive Dirac operator, though it may appear to be a natural discretization, 
is inadequate, since a property which corresponds to the elliptic estimate in the continuum theory is lost\footnote{In physics, the problem is known as fermion doubling since the naive Dirac operator develops multiple zero points in the momentum space. Due to these unphysical doubler zeros, 
the naive Dirac operator does not recover the elliptic estimate even in the continuum limit $a\to 0$.}.
One should instead employ, for instance, the Wilson Dirac operator, which we define below.
Recall that $a = 1/N$ is our lattice spacing and that $\gamma$ is the $\Z_2$-grading operator of $E$.

\begin{definition}
  The Wilson term $\Wilsonterm \colon \Gamma(\lE) \to \Gamma(\lE)$ is defined by
  \begin{equation}
    \Wilsonterm u := \frac{a}{2} \sum_{j=1}^d \forwarddifferenceoperator_j \big(\forwarddifferenceoperator_j\big)^* u
  \end{equation}
  for $u \in \Gamma(\lE).$
\end{definition}

\begin{definition}
  We define the Wilson Dirac operator $\WilsonDiracoperator \colon \Gamma(\lE) \to \Gamma(\lE)$ by $\WilsonDiracoperator := \naiveDiracoperator + \gamma \Wilsonterm$.
\end{definition}

Note that $W$ is selfadjoint and so is $\WilsonDiracoperator$.

There is a crucial property of the Wilson Dirac operator
which corresponds to the elliptic estimate in the continuum theory.
In \cite{Aoki:2024sjc}[Theorem 4.7], it was shown that there exist two positive 
$a$-independent constants $a_2>0$ and $C>0$ such that
the following inequality uniformly holds for any finite lattice spacing satisfying $0<a\le a_2$
and arbitrary $\phi \in \Gamma(\lattice[a]{E})$.
\begin{equation}
 \sum_{j=1}^d|| \forwarddifferenceoperator_j \phi||^2 \leq 2 || \WilsonDiracoperator_a \phi ||^2 +  C  || \phi||^2,
\end{equation}
where we have put the subscript $a$ to the Wilson Dirac operator to remind that it depends on the lattice spacing $a$.

Finally let us introduce the lattice versions of $L^2$ and $L^2_1$ norms ; 
\begin{align}
 ||v_a ||_{L^2} &:=  \Big( 
  a^d\sum_{ z \in \lattice[a]{X}} |v_a(z)|^2 \Big)^{1/2},\\
 ||v_a||_{L^2_1} &:= \left[||v_a ||^2_{L^2} + 
  \frac{a^d}{m_0^2}\sum_{ z \in \lattice[a]{X}} \sum_{i=1}^d |(\forwarddifferenceoperator_j v_a)(z)|^2\right]^{1/2}
\end{align}
for $v_a \in \Gamma(\lattice[a]{E})$. 
Since $\Gamma(\lattice[a]{E})$ is finite dimensional for fixed $a$, 
the completions with respect to the norms $||\cdot ||_{L^2}$ 
and $||\cdot ||_{L^2_1}$ coincide with $\Gamma(\lattice[a]{E})$ itself. 
But in the continuum limit $a\to 0$,
the two converge to the different continuum counterparts.

\section{From lattice to continuum}
\label{sec:from-lattice-to-continuum}

In this section, we recall a ``finite element method''
interpolating the lattice theory with the continuum theory,
which was discussed in our previous paper~\cite{Aoki:2024sjc}.

\subsection{Finite element interpolator}

We first define an operator
\begin{equation}
  \interpolator{a} \colon \Gamma(\lattice[a]{E}) \to \Gamma(E),
\end{equation}
which interpolates functions on the lattice and those on the continuum torus.
Using the translational symmetry of the metric, the difference of the coordinates
like $x-z$ is well-defined.

We first define a periodic function $\rho_a^{(1)}:\R \to \R$ with period $1$ (see Figure~\ref{fig:rho_a})
by $\rho_a^{(1)}(t):= \frac{1}{a}\max \{ 0, 1- t/a, 1-(1-t)/a\}$,
which induce functions $ \bar\rho_a:S^1\to \R$ and  $\rho_a:X\to \R$, 
\begin{equation}
\rho_a(x):=\prod_{i=1}^d \bar\rho_a(x_i) \quad  (x=(x_1, \ldots, x_d)\in X). 
\end{equation}
We will use this $\rho_a(x)$ as a cut-off function, which 
have the following properties.

\begin{figure*}[tbh]
  \centering
  \includegraphics[width=10cm]{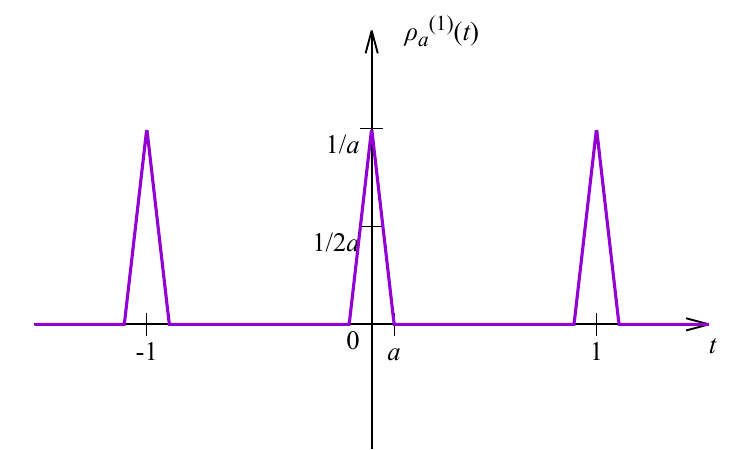}
  \caption{
    The function $\rho_a^{(1)}(t)$. 
      The function $\bar{\rho}_a(t)$ is its restriction to $t\in [0,1]$ where the two end points are identified.
  }
  \label{fig:rho_a}
\end{figure*}

\begin{enumerate}
 \item For arbitrary $x \in X$, we have 
\begin{equation}
 a^d\sum_{z\in \lattice[a]{X}
} \rho_a(x-z) =1.
\end{equation}
and for arbitrary $z \in \lattice[a]{X}$, we have 
\begin{equation}
 \int_{x \in X
} \rho_a(x-z) \;d x=1,
\end{equation}
where $d x$ is the volume form on $X$.
\item For $B=\{0 \} \cup \{ \pm \bm{e}_k \mid 1 \leq k \leq d\}$, we have 
\begin{equation}
 a^d\sum_{e \in  B } \int_{x  \in X } \rho_a(x) \rho_a(x-ae )\;d x  =1.
\end{equation}
\end{enumerate}

Using the above cut-off function, 
we define the map, the finite element interpolator,  $\interpolator{a} : \Gamma(\lattice[a]{E}) \to \Gamma(E)$ by,
\begin{equation}\label{eq:from-lat-to-cont}
 (\interpolator{a}\phi)(x) := a^d\sum_{z \in \lattice[a]{X}} \rho_a(x-z)\paralleltransport_{x,z} \phi(z) \quad (\phi\in  \Gamma(\lattice[a]{E}), \  x\in X)
\end{equation}
and its adjoint  $\iota_a^* :  \Gamma(E) \to \Gamma(\lattice[a]{E})$ by
\begin{equation}
 (\interpolator{a}^*\psi)(z):=\int_{x \in X} \rho_a(z-x) \paralleltransport_{x,z}^{-1}\psi(x)\;d x \quad (\psi\in \Gamma(E) , \ z\in \lattice[a]{X}).
\end{equation}
For later convenience, we also define
\begin{equation}
 \interpolator{a}|\phi(x)| := a^d\sum_{z \in \lattice[a]{X}} \rho_a(x-z)|\phi(z)|
\end{equation}
for $\phi\in \Gamma(\lattice[a]{E})$ and $x\in X$.

Note that $\rho_a(z-x)$ is nonzero only when $x$ is inside a unit hypercube
of the lattice where $z$ is one of the vertices.

\subsection{Properties}
\label{sec:properties}

We next review the key properties of the finite element interpolator $\interpolator{a}$.
One can see that neither $\interpolator{a} \interpolator{a}^*$ nor 
$\interpolator{a}^*\interpolator{a}$ is the identity map. 
This phenomenon is due to the nonlocal nature of these operators. 
Note, however, that $\rho_a(x-z)$ is nonzero only when $|x_i-z_i|<a$ for all $i$.
In the continuum limit $a\to 0$, they behave like the identity, 
which is guaranteed by the following properties proved in \cite{Aoki:2024sjc}.

\begin{enumerate}
 \item $\interpolator{a}$ and $\interpolator{a}^*$ have the same finite operator norm $||\interpolator{a}||=||\interpolator{a}^*||$.
 \item{}[Proposition 4.3 of \cite{Aoki:2024sjc}]\label{property:boundness} 
For $k=0,1$, the operator norm of 
$\interpolator{a} :L^2_k(\lE) \to L^2_k(E)$ is uniformly bounded with respect to $a$.
Likewise, the operator norm of 
$\interpolator{a}^* :L^2_k(E) \to L^2_k(\lE)$ is uniformly bounded with respect to $a$. 
Here we set $L_0^2(\cdot)=L^2(\cdot)$. 
\item{}[Proposition 4.4 of \cite{Aoki:2024sjc}] \label{property:iota*iota}
There exists $C>0$ which is independent of $a$ such that
for any $a$ and $\phi \in L^2(\lE)$
$$
|| \interpolator{a}^* \interpolator{a} \phi -\phi ||_{L^2}^2 \leq C a || \phi ||_{L^2_1}^2.
$$
\item{}[Proposition 4.5 of \cite{Aoki:2024sjc}] \label{property:iotaiota*}
For arbitrary $\psi \in \Gamma(E)\subset L^2(E)$, in the limit $a \to 0$,
 $\interpolator{a} \interpolator{a}^*\psi$ converges to $\psi$ in the strong $L^2$-sense:
$$ 
\interpolator{a} \interpolator{a}^* \psi \longrightarrow \psi .
$$
In particular, for arbitrary $\psi, \psi' \in \Gamma(E)\subset L^2(E)$, we have 
$$
\langle \interpolator{a}^* \psi',  \interpolator{a} ^* \psi \rangle \to \langle \psi', \psi \rangle
$$
in the $a\to 0$ limit.
\item{}[Proposition 4.6 of \cite{Aoki:2024sjc}] \label{property:Dirac}
For arbitrary $\psi \in \Gamma(E)\subset L^2(E)$, in the limit $a \to 0$, we have a strong $L^2$-convergence 
$$ 
\interpolator{a} (\WilsonDiracoperator)^* \interpolator{a}^* \psi \longrightarrow D^* \psi  \quad (L^2).
$$
Namely, for arbitrary $\psi, \psi' \in \Gamma(E)$, we have 
$$
\langle \interpolator{a}^* \psi' , (\WilsonDiracoperator)^* \interpolator{a} ^* \psi \rangle \to \langle \psi', D^*  \psi \rangle
$$
in the $a\to 0$ limit.   
\end{enumerate}

\section{domain-wall fermion Dirac operators and the Atiyah-Patodi-Singer index}
\label{sec:DWandAPS}

In this section, we will review the domain-wall fermion Dirac operator
and its relation to the Atiyah-Patodi-Singer index~\cite{MR4179728}.

\subsection{Dirac operators on manifolds with boundaries and the APS index}
\label{sec:APSindexreview}

We define $c$, $\epsilon$, and $\Gamma$ by
\begin{equation}\label{eq: clifford generators}
  c = \begin{pmatrix}0 & 1 \\ -1 & 0\end{pmatrix}, \quad
  \epsilon = \begin{pmatrix}0 & 1 \\ 1 & 0\end{pmatrix}, \quad \text{ and }
  \Gamma = \begin{pmatrix}1 & 0 \\ 0 & -1\end{pmatrix}.
\end{equation}
They satisfy that $c^2 = -1$, $\epsilon^2 = \Gamma^2 = 1$, $\Gamma = c\epsilon$, and they all anti-commute.

Let $X_+$ be an oriented even-dimensional Riemannian manifold with a boundary $Y$.
We assume that $Y$ has a collar neighbourhood isometric to the standard product
$I \times Y$ with an interval 
$I=[0,1)$ to which the coordinate is assigned by $u$,
where the boundary is located at $u=0$. 
Note that we do not assume the connectedness of $Y$.  
Let $S=S_+\oplus S_-$ be a $\Z_2$-graded Hermitian vector bundle on $X_+$ with
the $\Z_2$-grading operator $\gamma$ such that $\gamma|_{S_{\pm}}=\pm 1$.
We denote the space of the smooth sections\footnote{
In this section, we take a notation which specifies the base manifold.
} of $S$ on $X_+$ by $C^{\infty}(X_+;S)$.
Let $D \colon C^{\infty}(X_+;S) \to C^{\infty}(X_+;S)$ be a first-order, 
formally selfadjoint, elliptic partial differential operator.
We assume that $D$ is \emph{odd}
in the sense that it anti-commutes with $\gamma$.
We assume that $S$ and $D$ are standard in the following sense:
there exist a Hermitian vector bundle $E$ on $Y$ and a bundle isomorphism from $S|_{I \times Y}$ to $\C^2 \otimes E$ as $\Z_2$-graded Hermitian vector bundles such that, under this isomorphism, $D$ takes the form
\begin{equation}
  D = c \otimes \partial_u + \epsilon \otimes \DiracOperatorOnY = c \otimes \id
\left(\partial_u + B\right),\;\;\;
B:=  \begin{pmatrix}
    \DiracOperatorOnY & 0\\
    0 & -\DiracOperatorOnY
  \end{pmatrix},
\end{equation}
where the grading of $\C^2 \otimes E$ is given by $\Gamma\otimes {\rm id}$ and $\DiracOperatorOnY \colon C^{\infty}(Y;E) \to C^{\infty}(Y;E)$ is a formally selfadjoint, elliptic partial differential operator.
In this paper, we will concentrate on the case when $\DiracOperatorOnY$ has no zero eigenvalues, and assume this condition.
We call $B$ the boundary operator.

Let $C^{\infty}(X_+;S_{\pm}:P_{B}) := \big\{ f \in C^{\infty}(X_+; S_{\pm}) \mid P_{B}(f|_Y) = 0 \big\}$, where $P_{B} \colon L^2(Y;\C^2 \otimes E) \to L^2(Y;\C^2 \otimes E)$ denotes the spectral projection onto the span of the eigensections of the boundary operator ${B}$ with positive eigenvalues.
This is known as the APS boundary condition.
  We define the APS index \cite{MR4179728} of $D$ by
\begin{align}
  \APSindex(D) :=& \dim \big( \Ker D_+ \cap C^{\infty}(X_+;S_+:P_{B}) \big) \\
  &- \dim \big( \Ker D_- \cap C^{\infty}(X_+;S_-:P_{B}) \big),
\end{align}
where $D_\pm = DP_\pm$ with the projection operator $P_\pm=(\id\pm\gamma)/2$.
By the APS index theorem \cite{MR397797} the index is expressed by
\begin{equation}
 \APSindex(D) = \int_{X_+} a_d(S)(x) - \frac{1}{2}\eta(D_Y),
\end{equation}
where $a_d(S)(x)$ is a local invariant $d$-from given by 
curvature of $S$ (in this paper its explicit form is irrelevant) 
and $\eta(D_Y)$ is the Atiyah-Patodi-Singer $\eta$ invariant \cite{MR397797} of $D_Y$. 

\subsection{APS index without product structure near the boundary}
\label{sec:no-product}

In the original definition of the APS index given in the previous subsection, 
we assumed a product metric on $I\times Y$.
When the boundary does not have a product structure, 
we follow the prescription proposed in \cite{Gerd01011992,GILKEY1993129, braverman2019equivariant} (see also \cite{Ballmann_2012})
to define the index.

In general, the boundaries are not totally geodecic,
and it is nontrivial 
to identify an apporopriate boundary operator.
In \cite{Gerd01011992, GILKEY1993129, braverman2019equivariant}
the ``canonical'' boundary operator is defined by
\begin{equation}
B_c:= -(c\otimes \id) D -\partial_u + \frac{d_Y}{2}h,
\end{equation}
where $d_Y$ is dimension of $Y$ and $h$ denotes the mean curvature of $Y$ in $X$.
Here we have taken a trivialization where $\sigma(e_u)=c \otimes \id$ and 
the Clifford connection is trivial in the $u$ direction. 
We assume that $B_c$ is invertible.

Note that $B_c$ is a ``good'' operator, as it is purely intrinsic,
anticommutes with $c\otimes \id$, commutes with $\gamma$, 
and keeps $D$ selfadjoint by imposing the boundary condition 
by $P_{B_c} \colon L^2(Y;\C^2 \otimes E) \to L^2(Y;\C^2 \otimes E)$:
the spectral projection onto the span of the eigensections of the canonical 
boundary operator ${B_c}$ with positive eigenvalues.
The generalized APS index is then defined by
\begin{align}
  \gAPSindex(D) :=& \dim \big( \Ker D_+ \cap C^{\infty}(X_+;S_+:P_{B_c}) \big) \\
  &- \dim \big( \Ker D_- \cap C^{\infty}(X_+;S_-:P_{B_c}) \big).
\end{align}

With this $B_c$, the APS index theorem is modified to
\begin{equation}
 \gAPSindex(D) = \int_{X_+} a_d(S) - \frac{1}{2}\eta(D_Y) + \int_Y b_{d-1}(S),
\end{equation}
where the gauge invariant $(d-1)$-form $b_{d-1}(S)$
is given by the curvature of $S$ and the second fundamental form.
Interestingly, the additional boundary term can be obtained 
by difference from the standard APS index with a product metric near the boundary.
In Theorem 2.3 of \cite{Gerd01011992}, it was proved that the generalized APS index is
stable against perturbative deformations keeping $B_c$ unchanged.
  In Proposition~5.4 and appendix A of \cite{braverman2019equivariant}
  a deformation of $X_+$ we denote by $X^\epsilon_+$ 
  was explicitly given where the connection and metric in
  $[0,\epsilon/2)\times Y$ region of the collar part $I\times Y$ are smoothly deformed to have
    a product metric keeping $B_c$ unchanged, 
    while in the $[\epsilon/2, \epsilon)\times Y$ region they are smoothly connected
      to the rest of original $X_+$ at $u=\epsilon$.
      Let $D^\epsilon$ be the deformed Dirac operator on $X^\epsilon_+$ associated with the deformed metric and connection
      given in \cite{braverman2019equivariant}.
From the above theorem, there exists $\epsilon_0>0$ such that for arbitrary $\epsilon < \epsilon_0$,
\begin{equation}
\label{eq:gAPSAPS}
 \gAPSindex(D) = \APSindex(D^\epsilon),
\end{equation}
holds. 
Let $a^\epsilon_d(S)$ be the local invariant $d$-form on $X^\epsilon_+$.
Comparing with the standard APS index theorem with respect to $\APSindex(D^\epsilon)$ we obtain
\begin{equation}
 \int_Y b_{d-1}(S) = \lim_{\epsilon\to 0} \int_{X^\epsilon_+} a^\epsilon_d(S)-\int_{X_+} a_d(S),
\end{equation}
where $b_{d-1}(S)$ is the transgression form on $Y$.

Namely, the generalized APS index on $X_+$ can be interpreted as
the standard one on a slightly modified manifold $X^\epsilon_+$.

\subsection{Domain-wall fermion Dirac operators}

Let $X$ be a closed oriented even dimensional Riemannian manifold.
Let $S$  be a $\Z_2$-graded Hermitian vector bundle on $X$ with
the $\Z_2$-grading operator $\gamma$.
Let  $D \colon C^{\infty}(X;S) \to C^{\infty}(X;S)$ be an odd, first-order, 
formally selfadjoint, elliptic partial differential operator.
Let $Y \subset X$ be a separating submanifold that decomposes $X$ into the union of two compact manifolds
$X_+$ and $X_-$ which share the common boundary $Y$.
Let $\kappa: X\to [-1,1]$ be an $L^\infty$-function such that $\kappa \equiv \pm 1$ on  $X_{\pm} \setminus Y$.
We define the domain-wall fermion Dirac operator by
\begin{equation}
 D - m \kappa \gamma,
\end{equation}
where the mass parameter $m$ takes a positive value.

We also introduce a one-parameter family $\kappa_t$ with $t\in (-\infty,+\infty)$,
such that 
\begin{equation}
 \kappa_t =
\begin{cases}
\kappa \ (t\ge 1) \\ 
- 1\ (t\le -1)
\end{cases}
\end{equation}
where we assume that $\kappa_t$ is smooth with respect to $t$ in the range $I:=[-1,1]$.
The associated one-parameter family of the massive Dirac operator is given by
\begin{equation}
  \label{eq:Dt}
  D-m\kappa_t \gamma.
\end{equation}
Later assuming $D-m\kappa_t \gamma$ is invertible at $t=\pm1$, 
we consider the spectral flow
\begin{equation}
\mathrm{sf}[D-m\kappa_t \gamma] 
\end{equation}
which is an element of the $ K^1(I,\partial I)$ group,
where $\partial I$ is the boundary of $I$ which consists of two end points $\{-1,1\}$. 
The definitions of the group $K^1(I,\partial I)$ and the associated 
spectral flow will be given in the later sections.

In physics, it is known that the domain-wall fermion Dirac operator
has the so-called ``edge states'', which are low eigenstates whose amplitude is exponentially localized at
the domain-wall satisfying a different ``boundary condition''  
from the APS boundary condition introduced in the previous section.
In order to see this, let us consider $\R\times Y$ instead of a closed manifold $X$,
and a Hermitian vector bundle $E$ on $Y$. We denote the coordinate of $\R$ by $u$.
Let $\bar{S}=\C^2 \otimes E$ be a $\Z_2$-graded Hermitian vector bundle on $\R\times Y$
with the $\Z_2$-grading operator $\gamma=\Gamma \otimes \id_E$.
Let $\kappa_0 \colon \R \to \R$ be a sign function such that $\kappa_0(\pm u) = \pm 1$ for $u > 0$.
Let $\bar{D}\colon C^{\infty}(\R\times Y;\bar{S}) \to C^{\infty}(\R\times Y;\bar{S})$ 
be an odd, first-order, formally selfadjoint, elliptic partial 
differential operator.
The corresponding domain-wall fermion Dirac operator, which is essentially selfadjoint on
$L^2(\R \times Y; \C^2 \otimes E)$, takes the form
\begin{equation}
  \bar{D} - m\kappa_0\gamma = c \otimes \partial_u + \epsilon \otimes \DiracOperatorOnY - m\kappa_0\Gamma \otimes \id_E =
  \begin{pmatrix}
    -m \kappa_0\id_E& \partial_u + \DiracOperatorOnY \\
    -\partial_u + \DiracOperatorOnY & +m\kappa_0 \id_E
  \end{pmatrix},
\end{equation}
where $\DiracOperatorOnY \colon C^{\infty}(Y;E) \to C^{\infty}(Y;E)$ is a 
formally selfadjoint, elliptic partial differential operator.
The edge-states are defined as vectors in the 
intersection of the kernel of $c \otimes \partial_u - m\kappa_0\Gamma \otimes \id_E$ and the 
eigenspace of $\epsilon \otimes \DiracOperatorOnY$. 
Since $c \otimes \partial_u - m\kappa_0\Gamma \otimes \id_E$ and $\epsilon \otimes \DiracOperatorOnY$ 
anti-commute, each edge-state has the form  
\begin{equation}
 v_-\otimes \psi_{\lambda_\DiracOperatorOnY}\exp(-m|u|),
\end{equation}
for an eigensection $\psi_{\lambda_\DiracOperatorOnY}$ of $\DiracOperatorOnY$ with the eigenvalue $\lambda_\DiracOperatorOnY$ and 
$v_-=(1/\sqrt{2},-1/\sqrt{2})^T$ ($T$ denotes the transpose). 
The $L^2$-condition for edge-state implies $\epsilon v_-=-v_-$.
It is also important to note that the condition $\epsilon v_-=-v_-$
is different from the APS boundary condition.

\subsection{The domain-wall Dirac operators and the APS index}\label{eq:FFMOYY}

In this subsection, we discuss a nontrivial relation between the
domain-wall fermion Dirac operator and the APS index proved in \cite{MR4179728}.

We consider the same $X= X_+ \cup X_-$, the Hermitian vector bundle $S$
and the domain-wall fermion Dirac operator $D - m \kappa \gamma$
as in the previous subsection.
Here as in the subsection \ref{sec:APSindexreview}, 
let us assume a collar neighbourhood near the domain-wall with a product metric:
we assume the structure $I'\times Y$, where the interval
$I'$ is a double of $I$ so that $(I' \times Y)\cap X_\pm = I\times Y$.
We will consider the case without product structure later.
We also assume that $\DiracOperatorOnY$ has no zero eigenvalue.
In  \cite{MR4179728}, it was proved that there exists a constant $m_0>0$ that
depends only on $X,S$ and $D$ such that  
\begin{equation}\label{eq:FFMOYYformula}
  \APSindex (\restrictedTo{D}{X_+}) = -\frac{\eta(D - m \kappa \gamma) - \eta(D + m \gamma)}{2},
\end{equation}
holds for any $m>m_0$.
Note in this case both $D + m \gamma$ and $D - m \kappa \gamma$ are invertible.

The proof of the above equality was given by considering a certain embedding of $((-\infty, 0)\times Y) \cap X_+$ into 
$\R\times X$ where the coordinate for $\R$ is denoted by $t$.
We pull back the bundle $S$ on $X$ to $\R\times X$, which
will be denoted by the same symbol.
We introduce a selfadjoint operator $\DiracOperatorOnCylinderAPS_m \colon L^2(\R \times X; S \oplus S) \to L^2(\R \times X; S \oplus S)$ defined by
\begin{equation}
  \DiracOperatorOnCylinderAPS_m := c\otimes \partial_t + \epsilon \otimes (D-m\kappa_t\gamma)
=  \begin{pmatrix}
    0 & D-m\kappa_t\gamma + \partial_t \\
    D-m\kappa_t\gamma - \partial_t & 0
  \end{pmatrix}.
\end{equation}
Note that  $\DiracOperatorOnCylinderAPS_m$ is an odd operator with respect to the grading of 
$S\oplus S=\C^2\otimes S$ given by $\Gamma \otimes \id$.
In fact, the both sides of \eqref{eq:FFMOYYformula} correspond to
different evaluations of the same index of $\DiracOperatorOnCylinderAPS_m$.

One evaluation uses the localization of the zero eigensections which 
are localized in the neighbourhood of the submanifold in $\R \times X$ determined by $\kappa_t=0$, 
which is diffeomorphic to $((-\infty, 0)\times Y) \cap X_+$, 
when the mass parameter $m$ is sufficiently large.
Then the product formula of the index \cite{MR2361481} 
as well as \cite{MR397797}*{Proposition 3.11}
indicates that the index equals to the APS index on $X_+$:
\begin{equation}
  \textrm{Ind}(\DiracOperatorOnCylinderAPS_m) =  \APSindex (\restrictedTo{D}{X_+}).
\end{equation}

Another evaluation employs the APS index theorem,
where \cite{MR397797}*{Proposition 3.11} indicates that 
$\textrm{Ind}(\DiracOperatorOnCylinderAPS_m)$ equals to
the APS index on $I''\times X$ with two boundaries of 
the interval $I''$.
Note that the constant term in the asymptotic expansion of the heat kernel vanishes on such an odd-dimensional manifold as $I'' \times X$.
Therefore, the only boundary $\eta$-invariants contributes to the index, and we have 
\begin{equation}
 \textrm{Ind}(\DiracOperatorOnCylinderAPS_m) = -\frac{\eta(D - m \kappa \gamma) - \eta(D + m \gamma)}{2}.
\end{equation}

In this work, we use third evaluation using the spectral flow. 
Following the standard argument, 
we can show 
that $\textrm{dim}\textrm{Ker}(\DiracOperatorOnCylinderAPS_m \hat{P}_\pm)$ with $\hat{P}_\pm=(1\pm \Gamma \otimes \id)/2$
corresponds to the number of crossing zero eigenstates of $D-m\kappa_t \gamma$ from 
negative/positive to positive/negative, respectively. 
Thus we have the following modification of the theorem in \cite{MR4179728}.

\begin{theorem}\label{thm:FFMOYY}
  There exists a constant $m_0>0$ that depends only on $X,S$ and $D$ such that  
\begin{equation}\label{eq:FFMOYY}
 \APSindex (\restrictedTo{D}{X_+}) = \mathrm{sf}[D-m\kappa_t\gamma]
\end{equation}
holds for any $m>m_0$,
where $\mathrm{sf}$ denotes the spectral flow defined later in Section~\ref{sec:sf}.
\end{theorem}

Next let us consider the case where the neighborhood of $Y$ does not have a product metric.
In Sec.~\ref{sec:no-product}, we observed that we can deform $X_+$ 
to $X^\epsilon_+$ which has a product metric on its
collar near the boundary, without changing the boundary condition 
defined by the canonical boundary operator $B_c$.
The generalized APS index is stable against such deformations
and thus equals to the standard APS index of $D^\epsilon$ on $X^\epsilon_+$
defined in \cite{braverman2019equivariant} with a sufficiently small $\epsilon$.

We can deform $X$ to $X^\epsilon = X_+^\epsilon \cup X_-^\epsilon$ 
and the domain-wall fermion Dirac oparator to $D^\epsilon -m\kappa_t\gamma$ on it in the same way.
By Proposition 5.4 of \cite{braverman2019equivariant}, 
the deformed Clifford actions and connections depend continuously 
on the deformation parameter $\epsilon$.
Hence, after identifying the varying $L^2$-spaces by a continuous family of unitaries, 
the corresponding Dirac-type operators form a norm-continuous family $L^2_1(X,E)\to L^2(X,E)$.
Since we are considering the closed manifold and the domain of the Dirac-type operators is fixed, 
Proposition 2.2 of \cite{MR2114489} implies that this family is continuous in the Riesz topology.

Using this fact, we obtain the following theorem.
\begin{theorem}\label{thm:FFMOYYg}
  There exists a constant $m_0>0$ such that 
\begin{equation}\label{eq:FFMOYYg}
 \gAPSindex (\restrictedTo{D}{X_+}) 
 =\mathrm{sf}[D-m\kappa_t\gamma]
\end{equation}
holds for any $m>m_0$. 
\end{theorem}
\begin{proof}
Let $B_c$ the canonical boundary operator on $Y$.
Let $\lambda_{B_c}$ be the positive square root of the first non-zero eigenvalue of $B_c^2$.
We begin with two observations.
First, from the discussion in Sec.~\ref{sec:no-product}, there exists a constant $\epsilon_1>0$ such that $\gAPSindex (\restrictedTo{D}{X_+}) = \APSindex (\restrictedTo{D^{\epsilon}}{X^{\epsilon}_+})$ holds for any $0<\epsilon < \epsilon_1$.
Second, since $D$ and $D^\epsilon$ differ only in the $(-\epsilon, \epsilon)\times Y$ subregion of $X$, there exists a constant $\epsilon_2 > 0$ such that $\|D-D^{\epsilon}\| \le \lambda_{B_c}/100$ holds for any $0 < \epsilon < \epsilon_2$.
A closer examination of the proofs of Lemma 3 and Proposition 9 in \cite{MR4179728} shows that the constant $m_1$ in Proposition 9 depends only on $1/\epsilon$ and $\lambda_{B_c}$.
For $m > m_1$, we have an equality $\APSindex (\restrictedTo{D^{\epsilon}}{X^{\epsilon}_+})=\mathrm{sf}[D^{\epsilon}-m\kappa_t\gamma]$.
Note that $[D^{\epsilon}-m\kappa_t\gamma]$ is invertible at $t=0,1$.
Moreover, we have $\|(D^{\epsilon}-m\kappa\gamma)^{-1}\| \le 2/\lambda_{B_c}$.
For the last step of the proof, we define the following families of selfadjoint operators
\begin{equation}
D_1^t:= (1-t)D + t D^\epsilon + m\gamma,\;\;\; D_2^t := (1-t)D^\epsilon + t D - m\kappa \gamma.
\end{equation}
Since $(D_1^t)^2=[(1-t)D + t D^\epsilon]^2+m^2$, $D_1^t$ is invertible for any $t$.
For $D^t_2$, we have
\begin{equation}
D_2^t(D_2^0)^{-1}=\id + t(D-D^\epsilon)(D^\epsilon - m\kappa \gamma)^{-1}. 
\end{equation} 
We have
\begin{equation}
||(D-D^\epsilon)(D^\epsilon - m\kappa \gamma)^{-1}|| \le \frac{\lambda_{B_c}}{100} \cdot \frac{2}{\lambda_{B_c}} < 1,
\end{equation}
and therefore $D^t_2$ is invertible for any $t$.
Thus, 
\begin{equation}
\mathrm{sf}[D-m\kappa_t\gamma] = \mathrm{sf}[D_1^t]+\mathrm{sf}[D^\epsilon-m\kappa_t\gamma]+\mathrm{sf}[D_2^t]
\end{equation}
holds for $m>m_0$, where the first and third contributions are zero, and the claim follows.
\end{proof}

\subsection{domain-wall fermion Dirac operators on a flat torus}
In this work, we consider a $d$-dimensional square lattice 
with periodic or antiperiodic boundary conditions.
Therefore, the manifold $X$ we consider in the continuum limit 
is limited to a flat torus only.

The domain-wall between $X_+$ and $X_-$ can be any $d-1$-dimensional
curved submanifold of $X$.
With this curved domain-wall, nontrivial curvature contribution
(which corresponds to a gravitational background)
is induced as shown in 
\cite{Aoki:2022cwg,Aoki:2022aez,Aoki:2023lqp, Aoki:2024bwx}.
In our main theorem in which we compare the spectral flow of 
the lattice and continuum domain-wall fermion Dirac operators,
we do not require the product metric near the 
domain-wall between $X_+$ and $X_-$.
Therefore, our goal is to construct a lattice version of
$\mathrm{sf}[D-m\kappa_t \gamma]$ which equals to $\gAPSindex(D|_{X_+})$.


\section{\texorpdfstring{$K$}{K}-groups and unbounded selfadjoint operators}
\label{sec:K}


In this paper, we compare the spectral flow of continuum Dirac operators with that of lattice Dirac operators as elements of $K$-theory.
For this purpose, it is desirable to have a definition of $K$-groups and $KO$-groups that allows one to treat unbounded and bounded selfadjoint operators on an equal footing.
Accordingly, in this section, we present a self-contained definition of $K$-groups and $KO$-groups of arbitrary degree using unbounded selfadjoint operators, and prove that it is naturally isomorphic to the standard definition formulated in terms of bounded selfadjoint operators.
Although there are several possible choices of topology on the space of unbounded selfadjoint operators, we adopt the Riesz topology.
Related work in this direction includes that of Lesch~\cite{MR2114489}.
A detailed treatment of these topics will appear in a forthcoming paper.

We note that neither the suspension isomorphism nor Bott periodicity is proved in this section.
In later sections, we use the suspension isomorphism, e.g. \(K^0(\point) \cong K^1([0,1], \{0,1\})\), via the isomorphism established here by identifying it with the corresponding suspension isomorphism in the standard $K$-theory or $KO$-theory.

For notational simplicity, we explain only the definition of $K$-groups; the modification to $KO$-groups is straightforward.

We will construct two abelian semigroups, for each $p = 0,1,2,\dots$ and $q = -1,0,1,\dots$,
\begin{equation}
  K_{\bounded}^{p,q}(X,A) \text{ and } K_{\Riesz}^{p,q}(X,A)
\end{equation}
for a compact Hausdorff space $X$ and a closed subset $A \subset X$.
We will show in Section~\ref{subsection: group structure} that both semigroups are in fact groups.
$K_{\bounded}^{p,q}(X,A)$, defined using bounded selfadjoint operators, is naturally identified with the standard $K^n(X,A)$ for $n = p-q$.
$K_{\Riesz}^{p,q}(X,A)$ is defined using unbounded selfadjoint operators equipped with the Riesz topology.
In Theorem~\ref{theorem: isomorphism between bounded and Riesz}, we will prove that $K_{\bounded}^{p,q}(X,A)$ and $K_{\Riesz}^{p,q}(X,A)$ are canonically isomorphic.

\subsection{Definitions of \texorpdfstring{$K_{\bddorRiesz}^{p,q}$}{Kpq}-cocycles}\label{subsection: Kpq-cocycles}

We begin by recalling the definition of the Riesz topology.
Let $H$ be a separable Hilbert space over $\C$.
Let $\boundedoperators(H)$ denote the space of bounded operators on $H$ equipped with the norm topology $\operatornorm{\cdot}$.
Let $\boundedselfadjointoperators(H) \subset \boundedoperators(H)$ denote the space of bounded selfadjoint operators on $H$.
Let $\unboundedselfadjointoperators(H)$ denote the space of densely defined unbounded selfadjoint operators on $H$.

In this paper, all unbounded operators are assumed to be densely defined, and we use the term ``unbounded operator" to mean a possibly unbounded operator.
Moreover, the term ``Hilbert space" is understood to include finite-dimensional vector spaces equipped with a Hermitian inner product; the same convention applies to Hilbert bundles.

We define the Riesz transform $\Riesztransform \colon \unboundedselfadjointoperators(H) \to \boundedselfadjointoperators(H)$ via functional calculus associated with the homeomorphism
\begin{equation}
  \Riesztransform \colon \R \to (-1,1), \quad \lambda \mapsto \frac{\lambda}{\sqrt{1+\lambda^2}}.
\end{equation}
The Riesz topology on $\unboundedselfadjointoperators(H)$ is defined as the pullback of the norm topology on $\boundedselfadjointoperators(H)$ via $\Riesztransform$.
We note that $h \in \boundedselfadjointoperators(H)$ is in $\Image(\Riesztransform)$ if and only if $\operatornorm{h} \le 1$ and $\Kernel (h \pm \id_H) = \{0\}$.

Let $p = 0,1,2,\dots$ and $q=-1,0,1,\dots$.
The Clifford algebra $\cliffordalgebra^{q+1,p}$ is the $\ast$-algebra over $\R$ generated by
\begin{equation}
  \{\epsilon_0, \epsilon_1, \dots, \epsilon_q, e_1, e_2, \dots, e_p \}
\end{equation}
subject to the anticommutation relations
\begin{gather}
  \anticommutator{\epsilon_k}{\epsilon_{k'}} = 2 \delta_{k,k'}, \quad \anticommutator{e_l}{e_{l'}} = -2 \delta_{l,l'}, \\
  \anticommutator{\epsilon_k}{e_l} = 0, \\
  \epsilon_k^* = \epsilon_k, \quad e_l^* = - e_{l}
\end{gather}
for $k,k' = 0,1,\dots,q$ and $l,l' = 1,2,\dots,p$.
Note our convention for $p$ and $q$: $q$ starts from $-1$, and the order of $p$ and $q$ in the Clifford algebra $\cliffordalgebra^{q+1,p}$ is reversed.
We adopt the convention that, when $q=-1$, there are no generators $\epsilon_k$.
We denote by $\starHom(C_0,C_1)$ the set of $\ast$-homomorphisms between $\ast$-algebras $C_0$ and $C_1$.
\begin{definition}
  We define
  \begin{gather}
    \boundedselfadjointoperators^{p,q}(H) := \big\{ (c,h) \in \starHom\big(\cliffordalgebra^{q+1,p},\boundedoperators(H)\big) \times \boundedselfadjointoperators(H) \bigm\vert \heartsuit \big\} \\
    \unboundedselfadjointoperators^{p,q}(H) := \big\{ (c,h) \in \starHom\big(\cliffordalgebra^{q+1,p},\boundedoperators(H)\big) \times \unboundedselfadjointoperators(H) \bigm\vert \heartsuit \big\},
  \end{gather}
  where the condition $\heartsuit$ means that
  \begin{equation}
    \anticommutator{c(\epsilon_k)}{h} = 0, \quad \anticommutator{c(e_l)}{h} = 0
  \end{equation}
  for each $k = 0,1,\dots, q$ and $l = 1,2,\dots, p$.
\end{definition}

Let $X$ be a compact Hausdorff space.
Let $\mathcal{H} \to X$ be a Hilbert bundle over $X$ with fibre a separable Hilbert space over $\C$.
We assume throughout this paper that the structure group of a Hilbert bundle is equipped with the norm topology.
We denote by $\mathcal{H}_x$ the fibre of $\mathcal{H}$ at $x \in X$.

We would like to consider continuity for bundle maps between Hilbert bundles; however, since we work in an unbounded setting, we define it carefully as follows.
Let $x_0 \in X$.
Let $U$ be an open neighborhood of $x_0$ and $\phi \colon U \times \mathcal{H}_{x_0} \cong \left.\mathcal{H}\right|_U$ a local trivialisation of $\mathcal{H}$ over $U$.
For each $x \in U$, we let
\begin{equation}
  \phi^*_x \colon \boundedoperators(\mathcal{H}_x) \to \boundedoperators(\mathcal{H}_{x_0})
\end{equation}
be the bijection induced by $\phi$.

We first consider families of bounded operators.
Let $(c,h) = \big( \{c_x\}_{x \in X}, \{h_x\}_{x \in X} \big)$ be a family of $(c_x,h_x) \in \boundedselfadjointoperators^{p,q}(\mathcal{H}_x)$ parametrized by $x \in X$.
We define a map
\begin{equation}
  F_{\phi} \colon U \to \big( \boundedoperators(\mathcal{H}_{x_0}) \big)^{(q+1)+p+1}
\end{equation}
by
\begin{equation}
  x \mapsto \big( \phi^*_x(c_x(\epsilon_0)), \dots, \phi^*_x(c_x(\epsilon_q)), \phi^*_x(c_x(e_1)), \dots, \phi^*_x(c_x(e_p)), \phi^*_x(h_x) \big).
\end{equation}
\begin{definition}
  Let $(c,h)$ be a family as above.
  \begin{enumerate}
  \item
    Let $x_0 \in X$.
    We say that $(c,h)$ is continuous at $x_0$ if there exist an open neighborhood $U$ of $x_0$ and a local trivialisation $\phi$ of $\mathcal{H}$ over $U$ such that the map $F_{\phi} \colon U \to \big( \boundedoperators(\mathcal{H}_{x_0}) \big)^{(q+1)+p+1}$ defined as above is continuous at $x_0$ with respect to the norm topology on $\boundedoperators(\mathcal{H}_{x_0})$.
    Note that the continuity of $F_{\phi}$ at $x_0$ does not depend on the choice of $U$ and $\phi$.
  \item We say that $(c,h)$ is continuous if $(c,h)$ is continuous at each $x_0 \in X$.
  \end{enumerate}
\end{definition}

We next consider families of unbounded operators.
Let $(c,h) = \big( \{c_x\}_{x \in X}, \{h_x\}_{x \in X} \big)$ be a family of $(c_x,h_x) \in \unboundedselfadjointoperators^{p,q}(\mathcal{H}_x)$ parametrized by $x \in X$.
Recall that $h_x$ is unbounded.
We define a map
\begin{equation}
  F_{\phi}^{\Riesz} \colon U \to \big( \boundedoperators(\mathcal{H}_{x_0}) \big)^{(q+1)+p+1}
\end{equation}
by
\begin{equation}
  x \mapsto \big( \phi^*_x(c_x(\epsilon_0)), \dots, \phi^*_x(c_x(\epsilon_q)), \phi^*_x(c_x(e_1)), \dots, \phi^*_x(c_x(e_p)), \phi^*_x(\Riesztransform(h_x)) \big).
\end{equation}

\begin{definition}\label{definition: continuity of unbounded (c,h)}
  Let $(c,h)$ be an unbounded family as above.
  \begin{enumerate}
  \item
    Let $x_0 \in X$.
    We say that $(c,h)$ is Riesz-continuous at $x_0$ if there exist an open neighborhood $U$ of $x_0$ and a local trivialisation $\phi$ of $\mathcal{H}$ over $U$ such that the map $F_{\phi}^{\Riesz} \colon U \to \big( \boundedoperators(\mathcal{H}_{x_0}) \big)^{(q+1)+p+1}$ defined as above is continuous at $x_0$ with respect to the norm topology on $\boundedoperators(\mathcal{H}_{x_0})$.
    Note that the continuity of $F_{\phi}^{\Riesz}$ at $x_0$ does not depend on the choice of $U$ and $\phi$ because we equip the structure group of a Hilbert bundle with the norm topology.
  \item We say that $(c,h)$ is Riesz-continuous if $(c,h)$ is Riesz-continuous at each $x_0 \in X$.
  \end{enumerate}
\end{definition}

Having established these preliminaries, we now introduce the following spaces, which will be used to define $K$-cocycles.
\begin{definition}
  Let $X$ be a compact Hausdorff space and $\mathcal{H} \to X$ a Hilbert bundle over $X$.
  \begin{enumerate}
    \item We denote by $\boundedselfadjointoperators^{p,q}(\mathcal{H})$ the set of all bounded families $(c,h) = \big( \{c_x\}_{x \in X}, \{h_x\}_{x \in X} \big)$ with $(c_x,h_x) \in \boundedselfadjointoperators^{p,q}(\mathcal{H}_x)$ for each $x \in X$ that are continuous in the sense defined above.
    \item We denote by $\unboundedselfadjointoperators^{p,q}(\mathcal{H})$ the set of all unbounded families $(c,h) = \big( \{c_x\}_{x \in X}, \{h_x\}_{x \in X} \big)$ with $(c_x,h_x) \in \unboundedselfadjointoperators^{p,q}(\mathcal{H}_x)$ for each $x \in X$ that are Riesz-continuous in the sense defined above.
  \end{enumerate}
\end{definition}

We define $K$-cocycles as follows.
\begin{definition}
  Let $p = 0,1,2,\dots$ and $q = -1,0,1,\dots$.
  Let $X$ be a compact Hausdorff space and $A \subset X$ a closed subset.
  We use the symbol $\bddorRiesz$ to stand for either ``$\mathrm{Riesz}$'' or ``$\mathrm{bounded}$''.
  We say that a triple $\alpha = (\mathcal{H},c,h)$ is a $K_{\bddorRiesz}^{p,q}$-cocycle or simply $K$-cocycle if the following four properties are satisfied:
  \begin{enumerate}
    \item $\mathcal{H}$ is a Hilbert bundle over $X$ with fibre a separable Hilbert space.
    \item If $\bddorRiesz = \bounded$, then $(c,h) \in \boundedselfadjointoperators^{p,q}(\mathcal{H})$, and if $\bddorRiesz = \Riesz$, then $(c,h) \in \unboundedselfadjointoperators^{p,q}(\mathcal{H})$.
    \item $h_x$ is Fredholm for each $x \in X$.
    \item $\Kernel h_a = \{0\}$ for each $a \in A$.
  \end{enumerate}
  We write $\alpha \in \Kcocycle_{\bddorRiesz}^{p,q}(X,A)$ if $\alpha$ is a $K_{\bddorRiesz}^{p,q}$-cocycle.
\end{definition}

\begin{remark}
  Recall our convention that the term ``unbounded operator'' is used to mean a possibly unbounded operator.
  Accordingly, a family $h=\{h_x\}$ may consist of bounded operators at some points and unbounded operators at others, but this causes no difficulty for the notion of continuity defined above.
  Recall also our convention that the term ``Hilbert space'' is understood to include finite-dimensional vector spaces equipped with a Hermitian inner product.
  Consequently, the fibres of a Hilbert bundle may be finite-dimensional or infinite-dimensional, possibly varying from one connected component to another.

  In the definition of the $K$-group, only the behaviour of the spectrum near zero is essential, and the Fredholm condition precisely ensures that the spectrum near zero consists of isolated eigenvalues with finite multiplicity.
  By contrast, the distinction between bounded and unbounded operators concerns the behaviour of the spectrum at infinity, far away from zero, and is therefore completely inessential from the viewpoint of $K$-theory.
\end{remark}

\begin{remark}
  We note that the collection of all $K_{\bddorRiesz}^{p,q}$-cocycles does not form a set; thus, the expression $\alpha \in \Kcocycle_{\bddorRiesz}^{p,q}(X,A)$ is just a notation, and $\Kcocycle_{\bddorRiesz}^{p,q}(X,A)$ is not a set but just a formal expression.
  We will discuss set-theoretic issues in Section~\ref{subsection: set theory}.
\end{remark}

We regard a $K$-cocycle $\alpha = (\mathcal{H},c,h)$ as a Hilbert bundle equipped with additional structure.
The notions of isomorphism, pullback, and direct sum for $K$-cocycles are defined as follows.

\begin{definition}\label{definition: isom, pullback, sum}
  Let $p = 0,1,2,\dots$ and $q = -1,0,1,\dots$.
  \begin{enumerate}
    \item
      Let $X$ be a compact Hausdorff space and $A \subset X$ a closed subset.
      Let $\alpha = (\mathcal{H},c,h), \alpha' = (\mathcal{H}',c',h')\in \Kcocycle_{\bddorRiesz}^{p,q}(X,A)$.
      An isomorphism $\phi \colon \alpha \to \alpha'$ is a Hilbert bundle isomorphism $\phi \colon \mathcal{H} \to \mathcal{H}'$ such that $c = \phi^* c'$ and $h = \phi^* h'$.
      If there exists such an isomorphism, we denote $\alpha \cong \alpha'$.
    \item
      Let $X,Y$ be compact Hausdorff spaces and $A \subset X$ and $B \subset Y$ closed subsets.
      Let $f \colon X \to Y$ be a continuous map such that $f(A) \subset B$.
      For $\alpha = (\mathcal{H},c,h) \in \Kcocycle_{\bddorRiesz}^{p,q}(Y,B)$, the pullback $f^* \alpha$ is defined to be the triple $(f^*\mathcal{H}, f^*c, f^*h)$.
      Note that $f^* \alpha \in \Kcocycle_{\bddorRiesz}^{p,q}(X,A)$.
    \item
      Let $X$ be a compact Hausdorff space and $A,A' \subset X$ closed subsets.
      Let $\alpha = (\mathcal{H},c,h) \in \Kcocycle_{\bddorRiesz}^{p,q}(X,A)$ and $\alpha' = (\mathcal{H}','c',h') \in \Kcocycle_{\bddorRiesz}^{p,q}(X,A')$.
      The direct sum $\alpha \oplus \alpha'$ is defined to be $(\mathcal{H} \oplus \mathcal{H}', c \oplus c', h \oplus h')$.
      Note that $\alpha \oplus \alpha' \in \Kcocycle_{\bddorRiesz}^{p,q}(X, A \cap A')$.
  \end{enumerate}
\end{definition}

\subsection{Set-theoretic techniques}\label{subsection: set theory}

If the collection of $\Kcocycle_{\bddorRiesz}^{p,q}$-cocycles were a set, one could introduce an equivalence relation on it and define $K_{\bddorRiesz}^{p,q}$-groups as the corresponding quotient.
Since this is not the case, we proceed as follows.
This is a standard and frequently used technique.
In this paper, since we restrict ourselves to Hilbert bundles whose fibres are separable, we do not use the notion of universes.

\begin{definition}
  Let $X$ be a compact Hausdorff space and $A \subset X$ a closed subset.
  \begin{enumerate}
    \item
      Let $S$ be a set.
      We denote by $\Hilbertbundles(S)$ the \emph{set} of all Hilbert bundles over $X$ with fibre a separable Hilbert space whose total spaces are subsets of $S$.
    \item
      Let $S$ be a set.
      We denote by $\Kcocycle_{\bddorRiesz, S}^{p,q}(X,A)$ a set
      \begin{equation}
        \{ (\mathcal{H},c,h) \in \Kcocycle_{\bddorRiesz}^{p,q}(X,A) \mid \mathcal{H} \in \Hilbertbundles(S) \}.
      \end{equation}
    \item
      We denote by $\cong_S$ the relation on the set $\Kcocycle_{\bddorRiesz, S}^{p,q}(X,A)$ obtained by restricting the isomorphism $\cong$ of $\Kcocycle_{\bddorRiesz}^{p,q}$-cocycles.
      Note that $\cong_S$ is an equivalence relation.
      For $\alpha_S \in \Kcocycle_{\bddorRiesz, S}^{p,q}(X,A)$, we denote its equivalence class by $(\alpha_S \modulosettheory{S})$.
  \end{enumerate}
\end{definition}

We denote by $\cardinality{S}$ the cardinality of a set $S$ and by $\cardinality{\mathcal{H}}$ the cardinality of the total space of a Hilbert bundle $\mathcal{H}$.
We denote by $\ell^2(\Z)$ the standard separable Hilbert space whose standard basis is parametrized by $\Z$.
\begin{lemma}\label{lemma: set theoretic lemma}
  Let $X$ be a compact Hausdorff space and $A,A' \subset X$ be closed subsets.
  Let $S$ and $T$ be sets such that $\cardinality{S}, \cardinality{T} \ge \cardinality{X \times \ell^2(\Z)}$.
  \begin{enumerate}
    \item
      For any $\alpha \in \Kcocycle_{\bddorRiesz}^{p,q}(X,A)$, there exists $\alpha_S \in \Kcocycle_{\bddorRiesz,S}^{p,q}(X,A)$ such that $\alpha_S \cong \alpha$.
      Moreover, the equivalence class $(\alpha_S \modulosettheory{S}) \in \big( \Kcocycle_{\bddorRiesz,S}^{p,q}(X,A) / {\cong_S} \big)$ depends only on $\alpha$ and $S$.
    \item
      There exists a unique map
      \begin{equation}
        \oplus_S \colon \big( \Kcocycle_{\bddorRiesz,S}^{p,q}(X,A) / {\cong_S} \big) \times \big( \Kcocycle_{\bddorRiesz,S}^{p,q}(X,A') / {\cong_S} \big) \to \big( \Kcocycle_{\bddorRiesz,S}^{p,q}(X,A \cap A') / {\cong_S} \big)
      \end{equation}
      such that, for any $\alpha_S \in \Kcocycle_{\bddorRiesz,S}^{p,q}(X,A)$, $\beta_S \in \Kcocycle_{\bddorRiesz,S}^{p,q}(X,A')$, and $\gamma_S \in \Kcocycle_{\bddorRiesz,S}^{p,q}(X,A \cap A')$, we have
      \begin{equation}
        (\alpha_S \modulosettheory{S}) \oplus_S (\beta_S \modulosettheory{S}) = (\gamma_S \modulosettheory{S})
      \end{equation}
      if and only if $\alpha_S \oplus \beta_S \cong \gamma_S$ as $\Kcocycle_{\bddorRiesz}^{p,q}$-cocycles.
    \item
      There exists a unique bijection
      \begin{equation}
        \Phi_{T,S} \colon \big( \Kcocycle_{\bddorRiesz,S}^{p,q}(X,A) / {\cong_S} \big) \to \big( \Kcocycle_{\bddorRiesz,T}^{p,q}(X,A) / {\cong_T} \big)
      \end{equation}
      such that, for any $\alpha_S \in \Kcocycle_{\bddorRiesz,S}^{p,q}(X,A)$ and $\alpha_T \in \Kcocycle_{\bddorRiesz,T}^{p,q}(X,A)$, we have
      \begin{equation}
        \Phi_{T,S} \big( (\alpha_S \modulosettheory{S}) \big) = (\alpha_T \modulosettheory{T})
      \end{equation}
      if and only if $\alpha_S \cong \alpha_T$ as $\Kcocycle_{\bddorRiesz}^{p,q}$-cocycles.
    \item
      $\oplus_S$ and $\oplus_T$ are compatible with $\Phi_{T,S}$, that is, we have
      \begin{equation}
        \Phi_{T,S} \big( (\alpha_S \modulosettheory{S}) \oplus_S (\beta_S \modulosettheory{S}) \big) = \big( \Phi_{T,S} \big( (\alpha_S \modulosettheory{S}) \big) \big) \oplus_T \big( \Phi_{T,S} \big( (\beta_S \modulosettheory{S}) \big) \big)
      \end{equation}
      for any $\alpha_S \in \Kcocycle_{\bddorRiesz,S}^{p,q}(X,A)$ and $\beta_S \in \Kcocycle_{\bddorRiesz,S}^{p,q}(X,A')$.
  \end{enumerate}
\end{lemma}
\begin{proof}
  Let $\alpha \in \Kcocycle_{\bddorRiesz}^{p,q}(X,A)$.
  Since $\cardinality{S} \ge \cardinality{X \times \ell^2(\Z)} = \cardinality{\mathcal{H}}$, there exists an injection from the total space of $\mathcal{H}$ into $S$.
  This implies the first assertion.
  The remaining assertions follow immediately from the first one.
\end{proof}

The existence of the bijection $\Phi_{T,S}$ in the third assertion of Lemma~\ref{lemma: set theoretic lemma} suggests, at an informal level, that the set $\big( \Kcocycle_{\bddorRiesz,S}^{p,q}(X,A) / {\cong_S} \big)$ does not depend on the choice of $S$ provided that $\cardinality{S} \ge \cardinality{X \times \ell^2(\Z)}$.
To make this observation precise, we fix such a set $S$ for each $X$ as follows.
\begin{definition}
  Let $X$ be a compact Hausdorff space and $A,A' \subset X$ closed subsets.
  Set $S_X := X \times \ell^2(\Z)$.
  \begin{enumerate}
    \item
      We define
      \begin{equation}
        \big( \Kcocycle_{\bddorRiesz}^{p,q}(X,A)  / {\cong} \big) := \big( \Kcocycle_{\bddorRiesz,S_X}^{p,q}(X,A)  / {\cong_{S_X}} \big).
      \end{equation}
      Note that $\big( \Kcocycle_{\bddorRiesz}^{p,q}(X,A)  / {\cong} \big)$ is just a formal expression.
    \item
      We simply write $\oplus$ for $\oplus_{S_X}$, that is,
      \begin{equation}
        \oplus := \oplus_{S_X} \colon \big( \Kcocycle_{\bddorRiesz}^{p,q}(X,A) / {\cong} \big) \times \big( \Kcocycle_{\bddorRiesz}^{p,q}(X,A') / {\cong} \big) \to \big( \Kcocycle_{\bddorRiesz}^{p,q}(X,A \cap A') / {\cong} \big).
      \end{equation}
  \end{enumerate}
\end{definition}
\begin{definition}
  Let $X,Y$ be compact Hausdorff spaces and $A \subset X$ and $B \subset Y$ closed subsets.
  Let $f \colon X \to Y$ be a continuous map such that $f(A) \subset B$.
  We define a map
  \begin{equation}
    f^* \colon \big( \Kcocycle_{\bddorRiesz}^{p,q}(Y,B) / {\cong} \big) \to \big( \Kcocycle_{\bddorRiesz}^{p,q}(X,A) / {\cong} \big)
  \end{equation}
  by the pullback defined in Definition~\ref{definition: isom, pullback, sum}, which is well defined by Lemma~\ref{lemma: set theoretic lemma}.
  This map $f^*$ is characterized by the property that
  \begin{equation}
    f^*(\beta \modulosettheory{S_Y}) = (\alpha \modulosettheory{S_X})
  \end{equation}
  if and only if $f^* \beta \cong \alpha$ for any $\alpha \in \Kcocycle_{\bddorRiesz}^{p,q}(X,A)$ and $\beta \in \Kcocycle_{\bddorRiesz}^{p,q}(Y,B)$.
\end{definition}

For $\alpha \in \Kcocycle_{\bddorRiesz}^{p,q}(X,A)$, by the first assertion of Lemma~\ref{lemma: set theoretic lemma}, the equivalence class $\alpha_S \modulosettheory{S}$ depends only on $\alpha$; thus, by a slight abuse of language, we also write
\begin{equation}
  (\alpha \modulosettheory{}) := (\alpha_S \modulosettheory{S}) \in \big( \Kcocycle_{\bddorRiesz,S}^{p,q}(X,A)  / {\cong_S} \big),
\end{equation}
although, strictly speaking, $(\alpha \modulosettheory{})$ is just a formal expression.
Then, we can rephrase the definitions above simply as follows:
\begin{itemize}
  \item
    We define
    \begin{equation}
      (\alpha \modulosettheory{}) \oplus (\alpha' \modulosettheory{}) := (\alpha \oplus \alpha' \modulosettheory{}) \in \big( \Kcocycle_{\bddorRiesz}^{p,q}(X,A \cap A') / {\cong} \big)
    \end{equation}
    for any $\alpha \in \Kcocycle_{\bddorRiesz}^{p,q}(X,A)$ and $\alpha' \in \Kcocycle_{\bddorRiesz}^{p,q}(X,A')$.
  \item
    We define
    \begin{equation}
      f^*(\beta \modulosettheory{}) := \big( (f^*\beta) \modulosettheory{} \big) \in \big( \Kcocycle_{\bddorRiesz}^{p,q}(X,A)  / {\cong} \big)
    \end{equation}
    for any $\beta \in \Kcocycle_{\bddorRiesz}^{p,q}(Y,B)$.
\end{itemize}

\subsection{Definition of the semigroup \texorpdfstring{$K_{\bddorRiesz}^{p,q}$}{Kpq}}

With these preparations in place, we are finally in a position to define the semigroup $K_{\bddorRiesz}^{p,q}$.
Let $X$ be a compact Hausdorff space and $A \subset X$ a closed subset.
Let $i_0 \colon X \to X \times [0,1]$ and $i_1 \colon X \to X \times [0,1]$ be given by $i_t(x) = (x,t)$ for $t = 0,1$.
We use the symbol $\bddorRiesz$ to stand for either ``$\mathrm{Riesz}$'' or ``$\mathrm{bounded}$''.
Let $p = 0,1,2,\dots$ and $q=-1,0,1,\dots$.

We begin by defining a relation $\sim$ on the set $\big( \Kcocycle_{\bddorRiesz}^{p,q}(X,A) / {\cong} \big)$.
For $\alpha_0, \alpha_1 \in \Kcocycle_{\bddorRiesz}^{p,q}(X,A)$, we say that
\begin{equation}
  (\alpha_0 \modulosettheory{}) \sim (\alpha_1 \modulosettheory{})
\end{equation}
if and only if there exist $\beta_0, \beta_1 \in \Kcocycle_{\bddorRiesz}^{p,q}(X,X)$ and $\widetilde{\alpha} \in \Kcocycle_{\bddorRiesz}^{p,q}(X \times [0,1], A \times [0,1])$ such that
\begin{align}
  (\alpha_0 \modulosettheory{}) \oplus (\beta_0 \modulosettheory{}) & = i_0^* (\widetilde{\alpha} \modulosettheory{}) \\
  (\alpha_1 \modulosettheory{}) \oplus (\beta_1 \modulosettheory{}) & = i_1^* (\widetilde{\alpha} \modulosettheory{}).
\end{align}
It can be shown by standard arguments that $\sim$ is an equivalence relation.
Using $\sim$, we define the semigroup $K_{\bddorRiesz}^{p,q}$ as follows.
\begin{definition}
  We define
  \begin{equation}
    K_{\bddorRiesz}^{p,q}(X,A) := \big( \Kcocycle_{\bddorRiesz}^{p,q}(X,A) / {\cong} \big) / {\sim},
  \end{equation}
  and we write
  \begin{equation}
    [\alpha] := (\alpha \modulosettheory{}) / {\sim}
  \end{equation}
  for $\alpha \in \Kcocycle_{\bddorRiesz}^{p,q}(X,A)$.
\end{definition}

By definition, every element of $K_{\bddorRiesz}^{p,q}(X,A)$ is of the form $[\alpha]$ for some $\alpha \in \Kcocycle_{\bddorRiesz}^{p,q}(X,A)$.
Moreover, for $\alpha_0, \alpha_1 \in \Kcocycle_{\bddorRiesz}^{p,q}(X,A)$, we have $[\alpha_0] = [\alpha_1] \in K_{\bddorRiesz}^{p,q}(X,A)$ if and only if there exist $\beta_0, \beta_1 \in \Kcocycle_{\bddorRiesz}^{p,q}(X,X)$ and $\widetilde{\alpha} \in \Kcocycle_{\bddorRiesz}^{p,q}(X \times [0,1], A \times [0,1])$ such that $\alpha_0 \oplus \beta_0 \cong i_0^* \widetilde{\alpha}$ and $\alpha_1 \oplus \beta_1 \cong i_1^* \widetilde{\alpha}$.

By abuse of language, we call $\alpha_t \oplus \beta_t$ a stabilization of $\alpha_t$ for $t = 0,1$, and we say that $i_0^* \widetilde{\alpha}$ and $i_1^* \widetilde{\alpha}$ are homotopic.
With this convention, for $\alpha_0, \alpha_1 \in \Kcocycle_{\bddorRiesz}^{p,q}(X,A)$, we have $[\alpha_0] = [\alpha_1]$ if and only if they are homotopic after stabilization, in a way analogous to the standard construction.

The direct sum $\oplus$ on $\big( \Kcocycle_{\bddorRiesz}^{p,q}(X,A) / {\cong} \big)$ is clearly compatible with the equivalence relation $\sim$, and hence induces a commutative semigroup structure $+$ on $K_{\bddorRiesz}^{p,q}(X,A)$ by
\begin{equation}
  [\alpha] + [\alpha'] := [\alpha \oplus \alpha']
\end{equation}
for any $\alpha, \alpha' \in \Kcocycle_{\bddorRiesz}^{p,q}(X,A)$.
The identity is described as follows.
Let $\mathcal{O}_X$ be the trivial product Hilbert bundle $X \times \{0\}$.
For $(\mathcal{H},c,h) \in \Kcocycle_{\bddorRiesz}^{p,q}(X,A)$, if $\mathcal{H} = \mathcal{O}_X$, then $c$ and $h$ are uniquely determined.
We have
\begin{equation}
  [\alpha] + [\mathcal{O}_X,c,h] = [\alpha] \in K_{\bddorRiesz}^{p,q}(X,A)
\end{equation}
for any $\alpha \in \Kcocycle_{\bddorRiesz}^{p,q}(X,A)$.
We write $0$ for $[\mathcal{O}_X,c,h]$.

\subsection{Group structure on the semigroup \texorpdfstring{$K_{\bddorRiesz}^{p,q}$}{Kpq}}\label{subsection: group structure}

Let $X$ be a compact Hausdorff space and $A \subset X$ a closed subset.
In this subsection, we show that the abelian semigroup $K_{\bddorRiesz}^{p,q}(X,A)$ defined so far is in fact a group.

Let $p = 0,1,2,\dots$ and $q = -1,0,1,\dots$.
Recall that the Clifford algebra $\cliffordalgebra^{q+1,p}$ is the $\ast$-algebra over $\R$ generated by $\{\epsilon_0, \dots, \epsilon_q, e_1, \dots, e_p \}$ subject to the relations
\begin{equation}
  \anticommutator{\epsilon_k}{\epsilon_{k'}} = 2 \delta_{k,k'}, \quad \anticommutator{e_l}{e_{l'}} = -2 \delta_{l,l'}, \quad
  \anticommutator{\epsilon_k}{e_l} = 0, \quad
  \epsilon_k^* = \epsilon_k, \quad e_l^* = - e_{l}.
\end{equation}
We define a graded involution $\cliffordGamma \colon \cliffordalgebra^{q+1,p} \to \cliffordalgebra^{q+1,p}$ by
\begin{equation}
  \cliffordGamma(\epsilon_k) = -\epsilon_k, \quad \cliffordGamma(e_l) = -e_l
\end{equation}
for $k=0,\dots,q$ and $l = 1,\dots, p$.
For a $\ast$-algebra $\boundedoperators$ and a $\ast$-homomorphism $c \colon \cliffordalgebra^{q+1,p} \to \boundedoperators$, we define $-c := c \circ \cliffordGamma$; that is, $-c$ is a $\ast$-homomorphism characterized by
\begin{equation}
  (-c)(\epsilon_k) := - \big( c(\epsilon_k) \big), \quad (-c)(e_l) := - \big( c(e_l) \big)
\end{equation}
for $k=0,\dots,q$ and $l = 1,\dots, p$.
Note that $-c$ is not the pointwise negative of $c$; rather, the sign depends on the grading.
We define
\begin{equation}
  -\alpha := (\mathcal{H}, -c, -h) \in \Kcocycle_{\bddorRiesz}^{p,q}(X,A)
\end{equation}
for $\alpha = (\mathcal{H},c,h) \in \Kcocycle_{\bddorRiesz}^{p,q}(X,A)$.

The following proposition is not only a basic property of $K_{\bddorRiesz}^{p,q}(X,A)$ but also serves as a prototype for Proposition~\ref{proposition: key proposition}, the key argument of this paper.
Accordingly, we include a somewhat detailed explanation.
Such a complicated argument is required because the seemingly trivial continuity of the addition operator
\begin{equation}
  \unboundedselfadjointoperators \times \boundedselfadjointoperators \to \unboundedselfadjointoperators, \quad (A,B) \mapsto A+B
\end{equation}
becomes unexpectedly delicate when the space of unbounded selfadjoint operators $\unboundedselfadjointoperators$ is equipped with the Riesz topology, owing to the nonlinearity of the Riesz transform.
By contrast, arguments are much simpler if one works with the gap topology.
However, working with the gap topology would be somewhat excessive for our purposes, and therefore we choose in this paper to work with the more familiar Riesz topology.

\begin{proposition}\label{proposition: prototype}
  We have
  \begin{equation}
    [\alpha] + [-\alpha] = 0
  \end{equation}
  in $K_{\bddorRiesz}^{p,q}(X,A)$ for any $\alpha \in \Kcocycle_{\bddorRiesz}^{p,q}(X,A)$.
  In particular, $K_{\bddorRiesz}^{p,q}(X,A)$ is a group.
\end{proposition}

\begin{proof}
  Let $\alpha = (\mathcal{H},c,h) \in \Kcocycle_{\bddorRiesz}^{p,q}(X,A)$.
  We denote by $\pi \colon X \times [0,1] \to X$ the projection onto $X$.
  We define the triple $\widetilde{\alpha} = (\widetilde{\mathcal{H}},\widetilde{c},\widetilde{h})$ on $X \times [0,1]$ by setting
  \begin{equation}
    \widetilde{\mathcal{H}} := \pi^* \mathcal{H} \otimes_{\R} \R^2, \quad
    \widetilde{c}_{(x,t)} = \begin{pmatrix}
      c_x & 0 \\
      0 & -c_x
    \end{pmatrix}, \quad
    \widetilde{h}_{(x,t)} = \begin{pmatrix}
      h_x & t \\
      t & -h_x
    \end{pmatrix}
  \end{equation}
  for $(x,t) \in X \times [0,1]$.
  We first show that $\widetilde{\alpha} \in \Kcocycle_{\bddorRiesz}^{p,q}(X \times [0,1], A \times [0,1])$; that is, in either case $\bddorRiesz = \bounded \text{ or } \Riesz$, we verify that $\widetilde{c}$ and $\widetilde{h}$ depend continuously on $(x,t) \in X \times [0,1]$, that $\widetilde{h}_{(x,t)}$ is Fredholm for all $(x,t) \in X \times [0,1]$ and that $\Kernel \widetilde{h}_{(x,t)} = \{0\}$ for each $(x,t) \in A \times [0,1]$.

  The continuity of $\widetilde{c}$ follows directly from that of $c$.
  We show the continuity of $\widetilde{h}$ as follows:
  \begin{description}
    \item[$\bddorRiesz = \bounded$ case]
      The continuity of $\widetilde{h}$ follows directly from that of $h$.
    \item[$\bddorRiesz = \Riesz$ case]
      Recall that the Riesz transform is defined via $\Riesztransform(\lambda) = \lambda / \sqrt{1+\lambda^2}$.
      We have
      \begin{equation}
        \Riesztransform(\widetilde{h}_{(x,t)}) = \frac{\widetilde{h}_{(x,t)}}{\sqrt{1+\widetilde{h}_{(x,t)}^2}} = \frac{1}{\sqrt{1+t^2+h_x^2}}
        \begin{pmatrix}
          h_x & t \\
          t & -h_x
        \end{pmatrix}
      \end{equation}
      for each $(x,t) \in X \times [0,1]$.
      Then, we have, as $\lambda \to \pm \infty$,
      \begin{equation}
        \frac{1}{\sqrt{1+t^2+\lambda^2}} \begin{pmatrix}
          \lambda & t \\
          t & -\lambda
        \end{pmatrix}
        \longrightarrow \begin{pmatrix}
          \pm 1 & 0 \\
          0 & \mp 1
        \end{pmatrix}
      \end{equation}
      uniformly with respect to $t \in [0,1]$.
      Therefore, there exists a continuous map $\phi \colon [-1,1] \times [0,1] \to \matrixalgebra_2(\R)$ such that
      \begin{equation}
        \frac{1}{\sqrt{1+t^2+\lambda^2}} \begin{pmatrix}
          \lambda & t \\
          t & -\lambda
        \end{pmatrix} = \phi \big( \Riesztransform(\lambda), t \big)
      \end{equation}
      for each $(\lambda,t) \in \R \times [0,1]$.
      Consequently, we have
      \begin{equation}
        \Riesztransform(\widetilde{h}_{(x,t)}) = \phi \big( \Riesztransform(h_x), t \big)
      \end{equation}
      for each $(x,t) \in X \times [0,1]$.
      Thus, $\widetilde{h}$ depends continuously on $(x,t) \in X \times [0,1]$.
  \end{description}

  We next show that $\widetilde{h}_{(x,t)}$ is Fredholm for $(x,t) \in X \times [0,1]$ and that $\Kernel \widetilde{h}_{(x,t)} = \{0\}$ for $(x,t) \in A \times [0,1]$.
  Fix $(x,t) \in X \times [0,1]$.
  For $t \ne 0$, we have
  \begin{equation}
    \widetilde{h}_{(x,t)}^2 = \begin{pmatrix}
      h_x^2 + t^2 & 0 \\
      0 & h_x^2 + t^2
    \end{pmatrix} > 0.
  \end{equation}
  Hence, $\widetilde{h}_{(x,t)}$ is invertible.
  In particular, $\widetilde{h}_{(x,t)}$ is Fredholm.
  For $t = 0$, we have $\widetilde{h}_{(x,0)} = h_x \oplus (-h_x)$ so the Fredholm property of $\widetilde{h}_{(x,0)}$ and $\Kernel \widetilde{h}_{(x,0)} = \{0\}$ for $x \in A$ follow.

  Now we define
  \begin{gather}
    \alpha_0 := \alpha \oplus (-\alpha) \in \Kcocycle_{\bddorRiesz}^{p,q}(X,A), \quad \alpha_1 := 0 \in \Kcocycle_{\bddorRiesz}^{p,q}(X,A), \\
    \beta_0 := 0 \in \Kcocycle_{\bddorRiesz}^{p,q}(X,X), \quad \beta_1 := (\mathcal{H} \otimes_{\R} \R^2, c_1, h_1) \in \Kcocycle_{\bddorRiesz}^{p,q}(X,X),
  \end{gather}
  where
  \begin{equation}
    (c_1)_x := \begin{pmatrix}
      c_x & 0 \\
      0 & c_x
    \end{pmatrix}, \quad
    (h_1)_x := \begin{pmatrix}
      h_x & 1 \\
      1 & -h_x
    \end{pmatrix}
  \end{equation}
  for each $x \in X$.
  Then, we have
  \begin{equation}
    \alpha_0 \oplus \beta_0 \cong i_0^* \widetilde{\alpha}
    \text{ and }
    \alpha_1 \oplus \beta_1 \cong i_1^* \widetilde{\alpha}.
  \end{equation}
  Consequently, we have $[\alpha \oplus (-\alpha)] = [0]$.
  Thus, we have proved that $[\alpha] + [-\alpha] = 0$.
\end{proof}

\subsection{Degrees}

In this subsection, we explain how to assign degrees to $K$-groups,
thereby completing their definition.

\begin{proposition}
  Let $X$ be a compact Hausdorff space and $A \subset X$ a closed subset.
  There is a natural isomorphism
  \begin{equation}
    K_{\bddorRiesz}^{p,q}(X,A) \to K_{\bddorRiesz}^{p+1,q+1}(X,A)
  \end{equation}
  for each $p = 0,1,2,\dots$ and $q=-1,0,1,\dots$.
\end{proposition}
\begin{proof}
  We define $\epsilon_{\R^2}, \epsilon'_{\R^2}, e_{\R^2} \in \matrixalgebra_2(\R)$ by
  \begin{equation}
    \epsilon_{\R^2} := \begin{pmatrix}
      1 & 0 \\
      0 & 1
    \end{pmatrix}, \quad
    \epsilon'_{\R^2} := \begin{pmatrix}
      0 & 1 \\
      1 & 0
    \end{pmatrix}, \quad
    e_{\R^2} := \begin{pmatrix}
      0 & -1 \\
      1 & 0
    \end{pmatrix}.
  \end{equation}
  Then, $\epsilon_{\R^2}^2 = (\epsilon'_{\R^2})^2 = 1$ and $e_{\R^2}^2 = -1$.

  We first define a homomorphism
  \begin{equation}
    \delta \colon K_{\bddorRiesz}^{p,q}(X,A) \to K_{\bddorRiesz}^{p+1,q+1}(X,A).
  \end{equation}
  as follows.
  For $(\mathcal{H},c,h)\in \Kcocycle_{\bddorRiesz}^{p,q}(X,A)$, set $\delta([(\mathcal{H},c,h)]) =: [(\mathcal{H}',c',h')]$, where
  \begin{gather}
    \mathcal{H}' := \R^2 \otimes_{\R} \mathcal{H}, \\
    c'(\epsilon_k) := \epsilon_{\R^2} \otimes c(\epsilon_k), \quad c'(\epsilon_{q+1}) := \epsilon'_{\R^2} \otimes \id \\
    c'(e_l) := \epsilon_{\R^2} \otimes c(e_l), \quad c'(e_{p+1}) := e_{\R^2} \otimes \id, \\
    h' := \epsilon_{\R^2} \otimes h
  \end{gather}
  for $k = 0,1,\dots,q$ and $l = 1,2,\dots,p$.
  We next define a homomorphism
  \begin{equation}
    \rho \colon K_{\bddorRiesz}^{p+1,q+1}(X,A) \to K_{\bddorRiesz}^{p,q}(X,A).
  \end{equation}
  as follows.
  For $(\mathcal{H}',c',h')\in \Kcocycle_{\bddorRiesz}^{p+1,q+1}(X,A)$, set $\rho([(\mathcal{H}',c',h')]) =: [(\mathcal{H},c,h)]$, where
  \begin{gather}
    \mathcal{H} := \Kernel \big( c'(\epsilon_{q+1}) c'(e_{p+1}) - \id \otimes \id \big) \subset \mathcal{H}', \\
    c(\epsilon_k) := \left. c'(\epsilon_k) \right|_{\mathcal{H}}, \quad c(e_l) := \left. c'(e_l) \right|_{\mathcal{H}}, \\
    h := \left. h' \right|_{\mathcal{H}}
  \end{gather}
  for $k = 0,1,\dots,q$ and $l = 1,2,\dots,p$.
  Then, one can check that $\delta$ and $\rho$ are mutually inverse.
\end{proof}

This proposition gives rise to a direct system
\begin{equation}
  \dots \stackrel{\cong}{\to} K_{\bddorRiesz}^{p,q}(X,A) \stackrel{\cong}{\to} K_{\bddorRiesz}^{p+1,q+1}(X,A) \stackrel{\cong}{\to} K_{\bddorRiesz}^{p+2,q+2}(X,A) \stackrel{\cong}{\to} \dots
\end{equation}
and, using this system, we finally define the $K$-groups.

\begin{definition}
Let $X$ be a compact Hausdorff space and $A \subset X$ a closed subset.
  We define the abelian group $K_{\bddorRiesz}^n(X,A)$ as the direct limit
  \begin{equation}
    K_{\bddorRiesz}^n(X,A) := \varinjlim \big\{ \dots \to K_{\bddorRiesz}^{p,q}(X,A) \to K_{\bddorRiesz}^{p+1,q+1}(X,A) \to \cdots \big\}_{p-q = n}
  \end{equation}
for $n \in \Z$.
\end{definition}

\subsection{Riesz continuous families}

In this subsection, we present a criterion for Riesz continuity of families of unbounded selfadjoint operators, which will be used later.

We define the Riesz-continuity of a family of unbounded selfadjoint operators as follows.
See Definition~\ref{definition: continuity of unbounded (c,h)}.
Let $\mathcal{H}$ be a Hilbert bundle over a compact Hausdorff space $X$.
Let $x_0 \in X$, and let $U$ be an open neighbourhood of $x_0$ and $\phi \colon U \times \mathcal{H}_{x_0} \cong \left.\mathcal{H}\right|_U$ a local trivialisation of $\mathcal{H}$ over $U$.
For each $x \in U$, we let
\begin{equation}
  \phi^*_x \colon \unboundedselfadjointoperators(\mathcal{H}_x) \to \unboundedselfadjointoperators(\mathcal{H}_{x_0})
\end{equation}
be the bijection induced by $\phi$.
Let $h = \{h_x\}_{x \in X}$ be a family of $h_x \in \unboundedselfadjointoperators(\mathcal{H}_x)$ parametrized by $x \in X$.
We define a map $F_{\phi}^{\Riesz} \colon U \to \boundedoperators(\mathcal{H}_{x_0})$ by
\begin{equation}
  x \mapsto \Riesztransform (\phi^*_x h_x).
\end{equation}
We say that $h$ is Riesz-continuous at $x_0$ if $F_{\phi}^{\Riesz}$ is continuous at $x_0$ with respect to the norm topology on $\boundedoperators(\mathcal{H}_{x_0})$, and that $h$ is Riesz-continuous if $h$ is Riesz-continuous at each $x_0 \in X$.

The following proposition was shown implicitly by Lesch~\cite{MR2114489}*{PROPOSITION 2.2} using the theory of operator-monotonic increasing functions, the key point being the operator monotonicity of the square root $t \mapsto \sqrt{t}$.
We give an alternative argument based on the theory of \emph{operator-Lipschitz} functions.
The key observation here is that the Riesz transform is operator-Lipschitz.

For the theory of operator-Lipschitz functions, we refer the reader to the survey paper~\cite{MR3472849} and the references therein; for the reader’s convenience, we briefly collect here the results that are needed in the sequel.
A function $f \colon \R \to \R$ is called operator-Lipschitz if there exists a constant $C > 0$ such that
\begin{equation}
  \operatornorm{f(A)-f(B)} \le C \operatornorm{A-B}
\end{equation}
for any (possibly unbounded) selfadjoint operator $A$ and $B$ on a Hilbert space for which $A-B$ is bounded.
By \cite{MR3472849}*{Theorem 1.4.4}, a function $f \colon \R \to \R$ is operator-Lipschitz if $f$ belongs to the Besov space $B^1_{\infty,1}(\R)$.
The theory of Besov spaces is also neatly summarized in~\cite{MR3472849}*{Section 2.1}.
Now $\Riesztransform \in B^1_{\infty,1}(\R)$.
Thus, $\Riesztransform$ is operator-Lipschitz.

\begin{proposition}\label{proposition: Riesz continuity criterion}
  Let $X$ be a compact Hausdorff space.
  Let $\mathcal{H}$ be a Hilbert bundle over $X$.
  Let $h = \{h_x\}_{x \in X}$ be a family of unbounded selfadjoint operators $h_x \in \unboundedselfadjointoperators(\mathcal{H}_x)$ parametrized by $x \in X$.
  We assume that, for each $x_0 \in X$, there exist an open neighbourhood $U$ of $x_0$ and $\phi \colon U \times \mathcal{H}_{x_0} \to \left. \mathcal{H} \right|_U$ such that $\phi = \id$ at $x_0$, that the operator $\phi^*_x h_x - h_{x_0}$ extends to a bounded operator $m_x \in \boundedoperators(\mathcal{H}_{x_0})$ for each $x \in U$, and that the operator-valued map
  \begin{equation}
    U \ni x \mapsto m_x \in \boundedoperators(\mathcal{H}_{x_0})
  \end{equation}
  is continuous with respect to the norm topology on $\boundedoperators(\mathcal{H}_{x_0})$.
  Then, the family $\{h_x\}$ is continuous with respect to the Riesz topology.
\end{proposition}
\begin{proof}
  Fix $x_0 \in X$ and an open neighbourhood $U$ of $x_0$ as in the assumption of the proposition.
  Since $\Riesztransform$ is operator-Lipschitz, there exists a constant $C > 0$ such that
  \begin{equation}\label{equation: inequality form operator-Lipschitz function theory}
    \operatornorm{\Riesztransform(\phi^*_x h_x) - \Riesztransform(h_{x_0})} \le C \operatornorm{\phi^*_x h_x - h_{x_0}} = C \operatornorm{m_x}
  \end{equation}
  for any $x \in U$.
  Since $\phi = \id$ at $x_0$, we have $m_{x_0} = 0$.
  Moreover, $x \mapsto m_x$ is norm-continuous at $x_0$ by assumption.
  Thus, by the above inequality, $x \mapsto \Riesztransform(\phi^*_x h_x)$ is norm-continuous at $x_0$.
  Consequently, the family $\{h_x\}$ is Riesz-continuous.
\end{proof}

\subsection{The inverse element}

In this subsection, we explain a method for showing that $[\alpha \oplus (-\alpha)] = 0$ in the $K$-group for a $K$-cocycle $\alpha$.

We first fix the notation.
For a (possibly unbounded) selfadjoint operator $h$ on a Hilbert space $H$, we write its spectral decomposition as
\[
h = \int_{\R} \lambda \, dE_h(\lambda),
\]
where $E_h$ denotes the spectral measure associated with $h$.
We denote by $\spectrum(h)$ the spectrum of $h$.
We also define
\begin{gather}
  H_{h < -\lambda_0} := \Image E_h \big( \{ \lambda \in \R \mid \lambda < -\lambda_0 \} \big), \quad
  H_{h >  \lambda_0} := \Image E_h \big( \{ \lambda \in \R \mid \lambda > \lambda_0 \} \big), \\
  H_{\abs{h} < \lambda_0} := \Image E_h \big( \{ \lambda \in \R \mid \abs{\lambda} < \lambda_0 \} \big), \quad
  H_{\abs{h} > \lambda_0} := \Image E_h \big(\{ \lambda \in \R \mid \abs{\lambda} > \lambda_0 \} \big)
\end{gather}
for $\lambda_0 > 0$.
We adopt the same notation for families of selfadjoint operators on Hilbert bundles.

Let $\mathcal{H}$ be a Hilbert bundle over a compact Hausdorff space $X$ and $h$ be a Riesz-continuous family of unbounded selfadjoint operators.
Fix $x_0 \in X$.
For any $\Lambda > 0$, there exists $\lambda_0 \in (0,\Lambda)$ such that $\pm \lambda_0 \notin \spectrum(h_{x_0})$.
Fix such a $\lambda_0$.
Then, there exists an open neighbourhood $U_0$ of $x_0$ such that $\pm \lambda_0 \notin \spectrum(h_x)$ for any $x \in U_0$.
Thus, we obtain an orthogonal decomposition
\begin{equation}
  \left.\mathcal{H}\right|_{U_0} = \left(\left.\mathcal{H}\right|_{U_0}\right)_{h < -\lambda_0} \oplus \left(\left.\mathcal{H}\right|_{U_0}\right)_{\abs{h} < \lambda_0} \oplus \left(\left.\mathcal{H}\right|_{U_0}\right)_{h > \lambda_0}
\end{equation}
by the Riesz continuity of $h$.

The following proposition contains the key argument of this section; its prototype appears as Proposition~\ref{proposition: prototype}.
\begin{proposition}\label{proposition: key proposition}
  Let $X$ be a compact Hausdorff space and $A \subset X$ a closed subset.
  Let $p = 0,1,2,\dots$ and $q = -1,0,1,\dots$.
  Let $\alpha = (\mathcal{H},c,h), \alpha' = (\mathcal{H}',c',h') \in \Kcocycle_{\bddorRiesz}^{p,q}(X,A)$.
  Suppose that there exist $\lambda_0 > 0$ and a family $f = \{f_x\}_{x \in X}$ of bounded operators
  \begin{equation}
    f_x \colon \mathcal{H}_x \to \mathcal{H}'_x
  \end{equation}
  parametrized by $x \in X$ that satisfy the following conditions:
  \begin{enumerate}
    \item For each $x \in X$, the spectra $(\spectrum(h_x) \cap [-\lambda_0, \lambda_0])$ and $(\spectrum(h'_x) \cap [-\lambda_0, \lambda_0])$ consist only of isolated eigenvalues with finite multiplicity.
    \item The family $x \mapsto f_x$ is continuous with respect to the norm topology.
    \item For each $x \in X$, we have $f_x \circ c_x(g) = c'_x(g) \circ f_x$ for $g \in \cliffordalgebra^{q+1,p}$.
    \item For each $x \in X$, we have $f_x \circ h_x = h'_x \circ f_x$.
    \item
      For each $x \in X$, the composition
      \begin{equation}
        \left( E_{h'_x}([-\lambda_0,\lambda_0]) \right) \circ f_x \circ \left( E_{h_x}([-\lambda_0,\lambda_0]) \right) \colon \Image E_{h_x}([-\lambda_0,\lambda_0]) \to \Image E_{h'_x}([-\lambda_0,\lambda_0])
      \end{equation}
      is a unitary isomorphism.
  \end{enumerate}
  Then, we have $[\alpha' \oplus (-\alpha)] = 0 \in K_{\bddorRiesz}^{p,q}(X,A)$.
  In particular, $[\alpha'] = [\alpha] \in K_{\bddorRiesz}^{p,q}(X,A)$.
\end{proposition}
\begin{proof}
  Let $\alpha = (\mathcal{H},c,h), \alpha' = (\mathcal{H}',c',h') \in \Kcocycle_{\bddorRiesz}^{p,q}(X,A)$ satisfy the above assumptions.
  We will show that $\alpha' \oplus (-\alpha) \sim 0$.

  Let $\xi_{\lambda_0} \colon \R \to \R$ be an even continuous function (see Figure~\ref{fig:rho_a}) defined as
  \begin{equation}
    \xi_{\lambda_0} =
    \begin{cases}
      (1/3) \lambda_0 & \text{ if } \abs{\lambda} \le (1/3) \lambda_0 \\
      (2/3) \lambda_0 - \abs{\lambda} & \text{ if } (1/3) \lambda_0 < \abs{\lambda} < (2/3) \lambda_0 \\
      0 & \text{ if } \abs{\lambda} \ge (2/3) \lambda_0
    \end{cases}
  \end{equation}
  for $\lambda \in \R$, and set
  \begin{equation}
    (\psi_{h,h',\lambda_0})_x := \xi_{\lambda_0}(h'_x) \circ f_x \circ \xi_{\lambda_0}(h_x) \colon \mathcal{H}_x \to \mathcal{H}'_x
  \end{equation}
  for each $x \in X$.
  Then, $(\psi_{h,h',\lambda_0})_x$ is a finite-rank operator and $\operatornorm{(\psi_{h,h',\lambda_0})_x} \le 1/3$.

\begin{figure*}[tbh]
  \centering
  \includegraphics[width=10cm]{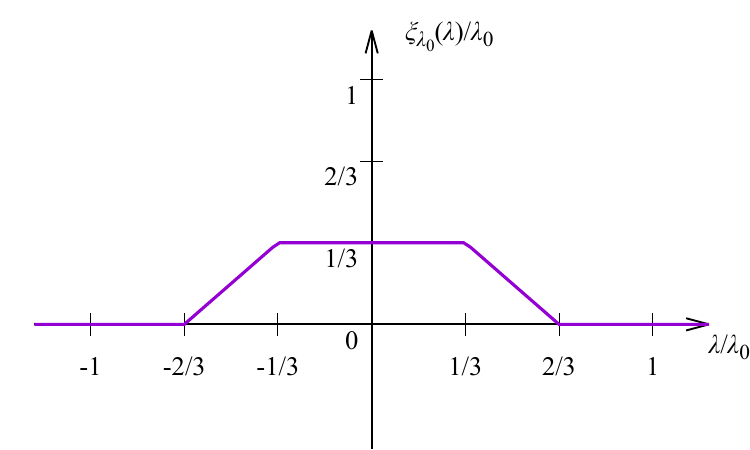}
  \caption{
 The function $\xi_{\lambda_0}(\lambda)$ normalized by $\lambda_0$.
  }
  \label{fig:rho_a}
\end{figure*}

  We define the triple $\widetilde{\alpha} = (\widetilde{\mathcal{H}},\widetilde{c},\widetilde{h})$ on $X \times [0,1]$ by setting
  \begin{equation}
    \widetilde{\mathcal{H}} := \pi^* \mathcal{H}' \oplus \pi^* \mathcal{H}, \quad
    \widetilde{c}_{(x,t)} = \begin{pmatrix}
      c'_x & 0 \\
      0 & -c_x
    \end{pmatrix}, \quad
    \widetilde{h}_{(x,t)} = \begin{pmatrix}
      h'_x & t (\psi_{h,h',\lambda_0})_x \\
      t (\psi_{h,h',\lambda_0})^*_x & -h_x
    \end{pmatrix}
  \end{equation}
  for $(x,t) \in X \times [0,1]$, where $\pi \colon X \times [0,1] \to X$ is the projection onto $X$.

  We first show that $\widetilde{\alpha} \in \Kcocycle_{\bddorRiesz}^{p,q}(X \times [0,1], A \times [0,1])$; that is, in either case $\bddorRiesz = \bounded \text{ or } \Riesz$, we verify that $\widetilde{c}$ and $\widetilde{h}$ depend continuously on $(x,t) \in X \times [0,1]$ and that $\widetilde{h}_{(x,t)}$ is Fredholm at each $(x,t) \in X \times [0,1]$.

  The continuity of $\widetilde{c}$ follows directly from that of $c$.
  We show the continuity of $\widetilde{h}$ as follows:
  \begin{description}
    \item[$\bddorRiesz = \bounded$ case]
      The continuity of $\widetilde{h}$ follows from those of $h$ and $t (\psi_{h,h',\lambda_0})$.
    \item[$\bddorRiesz = \Riesz$ case]
      Fix $(x_0,t_0) \in X \times [0,1]$.
      Let $\lambda_1$ be a constant that satisfies
      \begin{equation}
        \frac{2}{3}\lambda_0 < \lambda_1 < \lambda_0, \, \pm\lambda_1 \notin \spectrum(h_{x_0}), \text{ and } \pm\lambda_1 \notin \spectrum(h'_{x_0})
      \end{equation}
      For $t=0$, the Riesz-continuity of the family $(x,0) \mapsto \widetilde{h}_{(x,0)}$ follows directly from that of $x \mapsto h'_x$ and $x \mapsto -h_x$.
      Therefore, we have an orthogonal decomposition
      \begin{equation}
         \left.\mathcal{H}\right|_{U_0 \times [0,1]} = \left(\left.\mathcal{H}\right|_{U_0 \times [0,1]}\right)_{\widetilde{h}_{(x,0)} < -\lambda_1} \oplus \left(\left.\mathcal{H}\right|_{U_0 \times [0,1]
         }\right)_{\abs{\widetilde{h}_{(x,0)}} < \lambda_1} \oplus \left(\left.\mathcal{H}\right|_{U_0 \times [0,1]}\right)_{\widetilde{h}_{(x,0)} > \lambda_1}.
      \end{equation}
      By the choice of $\xi_{\lambda_0}$ and $\lambda_1$, for $t > 0$, the difference $\widetilde{h}_{(x,t)} - \widetilde{h}_{(x,0)}$ preserves the above decomposition.
      Hence, the Riesz transform $\Riesztransform(\widetilde{h}_{(x,t)})$ preserves the decomposition.
      Since its restriction on the second summand is a finite-rank operator, it is continuous at $(x_0,t_0)$.
      Moreover, since the difference $\widetilde{h}_{(x,t)} - \widetilde{h}_{(x,0)}$ acts trivially on the first and the third summands, the continuity of the restriction follows from  those of $\Riesztransform(h')$ and $\Riesztransform(h)$.
      Thus, $\widetilde{h}$ is Riesz-continuous at $(x_0,t_0)$.
  \end{description}

  We next show that $\widetilde{h}_{(x_0,t_0)}$ is Fredholm for each $(x_0,t_0) \in X \times [0,1]$.
  For $t_0 = 0$, the Fredholm property of $\widetilde{h}_{(x,t)}$ follows directly from that of $h'_x$ and $h_x$.
  For $t_0 \ne 0$, we consider the decomposition
  \begin{equation}
    \mathcal{H}_{(x_0,t_0)} = \left(\mathcal{H}_{(x_0,t_0)}\right)_{\widetilde{h}_{(x_0,t_0)} < -\lambda_1} \oplus \left(\mathcal{H}_{(x_0,t_0)}\right)_{\abs{\widetilde{h}_{(x_0,t_0)}} < \lambda_1} \oplus \left(\mathcal{H}_{(x_0,t_0)}\right)_{\widetilde{h}_{(x_0,t_0)} > \lambda_1}.
  \end{equation}
  It suffices to verify the Fredholm property on each summand.
  On the first and third summands, we have $\widetilde{h}_{(x_0,t_0)} = h'_{x_0} \oplus (-h_{x_0})$, which is Fredholm.
  The restriction to the second summand is a finite-rank operator, hence is Fredholm.

  We next show that $\Kernel \widetilde{h}_{(x_0,t_0)} = \{0\}$ for each $(x_0,t_0) \in A \times [0,1]$.
  For $t_0 = 0$, we have $\widetilde{h}_{x_0,0} = h'_{x_0} \oplus (-h_{x_0})$, and, hence, $\Kernel \widetilde{h}_{(x_0,0)} = \{0\}$.
  Suppose that $t_0 \ne 0$.
  Let $v' \oplus v \in \Kernel \widetilde{h}_{(x_0,t_0)}$.
  Then, we have
  \begin{gather}
    h'_{x_0} v' + t_0 (\psi_{h,h',\lambda_0})_{x_0} v = 0,\\
    t_0 (\psi_{h,h',\lambda_0})^*_{x_0} v' - h_{x_0} v = 0.
  \end{gather}
  Hence, we have
  \begin{gather}
     (h'_{x_0})^2 v' + t_0 h'_{x_0} (\psi_{h,h',\lambda_0})_{x_0} v = 0, \\
    t_0  (\psi_{h,h',\lambda_0})_{x_0} (\psi_{h,h',\lambda_0})^*_{x_0} v' -  (\psi_{h,h',\lambda_0})_{x_0} h_{x_0} v = 0.
  \end{gather}
  Since $(\psi_{h,h',\lambda_0})_{x_0} \circ h_{x_0} = h'_{x_0} \circ (\psi_{h,h',\lambda_0})_{x_0}$, the second equation implies
  \begin{equation}
    t_0 h'_{x_0} (\psi_{h,h',\lambda_0})_{x_0} v = t_0 (\psi_{h,h',\lambda_0})_{x_0} h_{x_0} v = t_0^2  (\psi_{h,h',\lambda_0})_{x_0} (\psi_{h,h',\lambda_0})^*_{x_0} v'.
  \end{equation}
  Hence, we have
  \begin{equation}
    (h'_{x_0})^2 v' + t_0^2  (\psi_{h,h',\lambda_0})_{x_0} (\psi_{h,h',\lambda_0})^*_{x_0} v' = 0.
  \end{equation}
  Thus, we have $h'_{x_0} v' = 0$ and $(\psi_{h,h',\lambda_0})^*_{x_0} v' = 0$.
  Then, $h'_{x_0} v' = 0$ implies
  \begin{equation}
     v' \in E_{h'_{x_0}}(\{0\}).
  \end{equation}
  On the other hand, since $(\psi_{h,h',\lambda_0})^*_{x_0} = \xi_{\lambda_0}(h) \circ f^*_{x_0} \circ \xi_{\lambda_0}(h')$ and $f^*$ is an isomorphism, $(\psi_{h,h',\lambda_0})^*_{x_0} v' = 0$ implies
  \begin{equation}
    v' \in E_{h'_{x_0}} \big( \{\lambda \in \R \mid \xi_{\lambda_0} = 0 \} \big).
  \end{equation}
  By the definition of $\xi_{\lambda_0}$, we have $\{0\} \cap \{\lambda \in \R \mid \xi_{\lambda_0} = 0 \} = \emptyset$.
  Thus, $v' \in E_{h'_{x_0}}(\emptyset) = \{0\}$.
  In the same way, we have $v = 0$.
  Consequently, we have proved that $\Kernel \widetilde{h}_{(x_0,t_0)} = \{0\}$.

  Now we define
  \begin{gather}
    \alpha_0 := \alpha' \oplus (-\alpha) \in \Kcocycle_{\bddorRiesz}^{p,q}(X,A), \quad \alpha_1 := 0 \in \Kcocycle_{\bddorRiesz}^{p,q}(X,A) \\
    \beta_0 := 0 \in \Kcocycle_{\bddorRiesz}^{p,q}(X,X), \quad \beta_1 := (\mathcal{H}' \oplus \mathcal{H}, c_1, h_1) \in \Kcocycle_{\bddorRiesz}^{p,q}(X,X),
  \end{gather}
  where
  \begin{equation}
    (c_1)_x := \begin{pmatrix}
      c'_x & 0 \\
      0 & -c_x
    \end{pmatrix}, \quad
    (h_1)_x := \begin{pmatrix}
      h'_x & 1 \\
      1 & -h_x
    \end{pmatrix}
  \end{equation}
  for each $x \in X$.
  Then, we have $\alpha_0 \oplus \beta_0 \cong i_0^* \widetilde{\alpha}$ and $\alpha_1 \oplus \beta_1 \cong i_1^* \widetilde{\alpha}$.
  Consequently, we have $[\alpha' \oplus (-\alpha)] = [0]$.
  Thus, we have proved that $[\alpha] + [-\alpha] = 0$.
\end{proof}

By combining Propositions~\ref{proposition: Riesz continuity criterion} and Proposition~\ref{proposition: key proposition}, we obtain a criterion for two $K$-cocycles defined by unbounded selfadjoint operators to represent the same element in the $K$-group.
This criterion will be used in the proof of the main theorem (Theorem~\ref{goal}).

\begin{theorem}\label{theorem: minus-variant2}
  Let $X$ be a compact Hausdorff space and $A \subset X$ a closed subset.
  Let $p=0,1,2,\dots$ and $q=-1,0,1,\dots$.
  Let $\alpha = (\mathcal{H},c,h), \alpha' = (\mathcal{H}',c',h') \in \Kcocycle_{\Riesz}^{p,q}(X,A)$.
  Let $f = \{f_x\}_{x \in X}$ be a family of bounded operators $f_x \colon \mathcal{H}_x \to \mathcal{H}'_x$ parametrized by $x \in X$.
  We assume the following conditions:
  \begin{enumerate}
    \item
      For each $x_0 \in X$, there exist an open neighbourhood $U$ of $x_0$ and $\phi \colon U \times \mathcal{H}_{x_0} \to \left. \mathcal{H} \right|_U$ such that $\phi = \id$ at $x_0$, that the operator $\phi^*_x h_x - h_{x_0}$ extends to a bounded operator $m_x \in \boundedoperators(\mathcal{H}_{x_0})$ for each $x \in U$.
      Moreover, the operator-valued map $U \ni x \mapsto m_x \in \boundedoperators(\mathcal{H}_{x_0})$ is continuous with respect to the norm topology on $\boundedoperators(\mathcal{H}_{x_0})$.
    \item We impose the same assumption on $\alpha'$.
    \item For each $x \in X$, we have $f_x \circ c_x(g) = c'_x(g) \circ f_x$ for $g \in \cliffordalgebra^{q+1,p}$.
    \item
      We have
      \begin{equation}
        \Kernel \begin{pmatrix}
          h'_x & f_x \\
          f^*_x & -h_x
        \end{pmatrix}
        = 0
      \end{equation}
      for each $x \in X$.
  \end{enumerate}
  Then, we have $[\alpha' \oplus (-\alpha)] = 0 \in K_{\bddorRiesz}^{p,q}(X,A)$.
  In particular, $[\alpha'] = [\alpha] \in K_{\bddorRiesz}^{p,q}(X,A)$.
\end{theorem}
\begin{proof}
  Let $\widetilde{\mathcal{H}} := (\mathcal{H} \oplus \mathcal{H}') \times [0,1]$.
  For each $(x,t) \in X \times [0,1]$, we set $\widetilde{c}_{(x,t)} := c'_x \oplus c_x$ and
  \begin{equation}
    \widetilde{h}_{(x,t)} := \begin{pmatrix}
      h'_x & tf_x \\
      tf^*_x & -h_x
    \end{pmatrix},
  \end{equation}
  and let $\widetilde{c} = \{\widetilde{c}_{(x,t)}\}_{(x,t) \in X \times [0,1]}$ and $\widetilde{h} = \{\widetilde{h}_{(x,t)}\}_{(x,t) \in X \times [0,1]}$.
  By Proposition~\ref{proposition: Riesz continuity criterion}, the family $h$ is Riesz-continuous.
  Hence, $\widetilde{\alpha} := (\widetilde{\mathcal{H}},\widetilde{c},\widetilde{h}) \in \Kcocycle_{\Riesz}^{p,q}(X \times [0,1], A \times [0,1])$.
  Then, using $\widetilde{\alpha}$, we can argue exactly as in the proof of
  Proposition~\ref{proposition: key proposition} to conclude that $[\alpha' \oplus (-\alpha)] = 0 \in K_{\bddorRiesz}^{p,q}(X,A)$.
\end{proof}

\subsection{Isomorphism \texorpdfstring{$K_{\bounded}^{\ast} \cong K_{\Riesz}^{\ast}$}{between {Kpq}-groups}}\label{subsection: isomorphism}

Let $X$ be a compact Hausdorff space and $A \subset X$ a closed subset.
Let $p=0,1,2,\dots$ and $q=-1,0,1,\dots$.
In this subsection, we finally establish the isomorphism $K_{\bounded}^{p,q}(X,A) \cong K_{\Riesz}^{p,q}(X,A)$.

We note that, if $f \colon \R \to \R$ is a bounded continuous function and $h = \{h_x\}_{x \in X}$ is a norm-continuous family of bounded selfadjoint operators $h_x \in \boundedselfadjointoperators(\mathcal{H}_x)$ on a Hilbert bundle $\mathcal{H}$ over $X$, then the family $f(h) := \{f(h_x)\}$ is also norm-continuous.

We define a homomorphism
\begin{equation}
  \bddtoRiesz \colon K_{\bounded}^{p,q}(X,A) \to K_{\Riesz}^{p,q}(X,A)
\end{equation}
by simply sending $[\alpha] \mapsto [\alpha]$ for $\alpha \in \Kcocycle_{\bounded}^{p,q}(X,A)$, which is clearly well defined.

To define the inverse, we introduce an auxiliary function $\TRtohalf \colon \R \to \R$ by setting
\begin{equation}
  \TRtohalf(\lambda) := \begin{cases}
    -\frac{1}{2} & \text{ if } \lambda \le -\frac{1}{2} \\
    \lambda & \text{ if } \abs{\lambda} < \frac{1}{2} \\
    \frac{1}{2} & \text{ if } \lambda \ge \frac{1}{2}
  \end{cases}
\end{equation}
for $\lambda \in \R$.
This function is bounded and continuous.
We then define a homomorphism
\begin{equation}
  \Riesztobdd \colon K_{\Riesz}^{p,q}(X,A) \to K_{\bounded}^{p,q}(X,A)
\end{equation}
by setting
\begin{equation}
  [(\mathcal{H}',c',h')] \mapsto [(\mathcal{H}',c',\TRtohalf(h'))]
\end{equation}
for $(\mathcal{H}',c',h') \in \Kcocycle_{\Riesz}^{p,q}(X,A)$.
The well-definedness of this homomorphism follows from the existence of a continuous function $\Tunittohalf \colon [-1,1] \to [-1/2,1/2]$ such that $\TRtohalf = \Tunittohalf \circ \Riesztransform$.

\begin{figure*}[tbh]
  \centering
  \includegraphics[width=10cm]{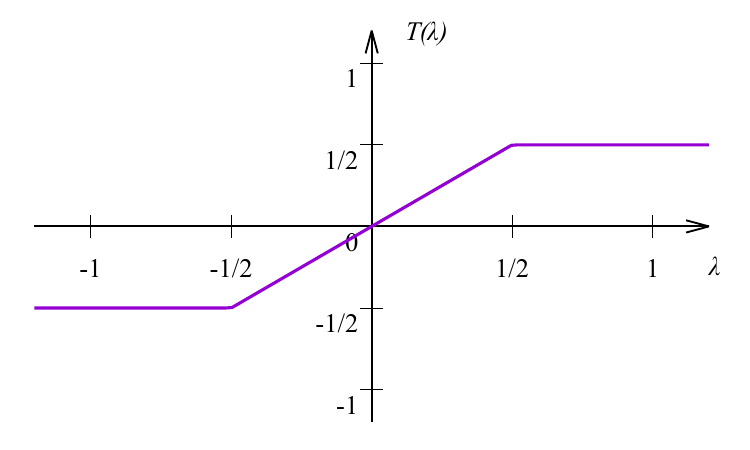}
  \caption{
 The function $T_{\R \to [-1/2,1/2]}(\lambda)$.
  }
  \label{fig:T}
\end{figure*}

The homomorphisms $\bddtoRiesz$ and $\Riesztobdd$ are compatible with the direct systems $\{K_{\bddorRiesz}^{p,q}(X,A) \to K_{\bddorRiesz}^{p+1,q+1}(X,A)\}_{p-q=n}$.
Thus, for each $n \in \Z$, they induce homomorphisms
\begin{gather}
  \bddtoRiesz \colon K_{\bounded}^n(X,A) \to K_{\Riesz}^n(X,A) \\
  \Riesztobdd \colon K_{\Riesz}^n(X,A) \to K_{\bounded}^n(X,A),
\end{gather}
which we denote by the same notation.

\begin{theorem}\label{theorem: isomorphism between bounded and Riesz}
  Let $X$ be a compact Hausdorff space and $A \subset X$ a closed subset.
  For each $n \in \Z$, the homomorphisms $\bddtoRiesz$ and $\Riesztobdd$ between $K_{\Riesz}^n(X,A)$ and $K_{\bounded}^n(X,A)$ are mutually inverse.
  In particular, we have a natural isomorphism
  \begin{equation}
    K_{\Riesz}^n(X,A) \cong K_{\bounded}^n(X,A)
  \end{equation}
  for all $n \in \Z$.
\end{theorem}
The proof relies on Proposition~\ref{proposition: key proposition}, which provides the key argument in this paper.
\begin{proof}
  Let $n \in \Z$, and fix $p=0,1,2,\dots$ and $q=-1,0,1,\dots$ such that $p-q = n$.
  We first prove that $\Riesztobdd \circ \bddtoRiesz = \id$.
  Let $\alpha = (\mathcal{H},c,h) \in \Kcocycle_{\bounded}^{p,q}(X,A)$, and set $\alpha' := (\mathcal{H},c,\TRtohalf(h))$.
  Since $X$ is compact and $h$ is a continuous family, there exists $\lambda_0 \in (0,1/2)$ such that $(\spectrum(h_x) \cap [-\lambda_0,\lambda_0])$ consists only of isolated eigenvalues with finite multiplicity for each $x \in X$.
  Fix such a $\lambda_0$ and let $f := \id$.
  Then, by Proposition~\ref{proposition: key proposition}, we obtain $[\alpha] = [\alpha']$.

  We next prove that $\bddtoRiesz \circ \Riesztobdd = \id$.
  Let $\beta := (\mathcal{H}',c',h') \in \Kcocycle_{\Riesz}^{p,q}(X,A)$, and set $\beta' := (\mathcal{H}',c',\TRtohalf(h'))$.
  Fix again $\lambda_0 \in (0,1/2)$ such that $(\spectrum(h'_x) \cap [-\lambda_0,\lambda_0])$ consists only of isolated eigenvalues with finite multiplicity for each $x \in X$, and let $f := \id$.
  Then, by Proposition~\ref{proposition: key proposition}, we obtain $[\beta] = [\beta']$.
\end{proof}

This completes our construction of $K$-groups.

\subsection{Gap topology}

In this final subsection, although somewhat tangential to the main line of argument, we state a proposition that plays an essential role in defining $K$-groups using the gap topology.
The use of the gap topology will be necessary when we generalize our main theorem (Theorem~\ref{goal}) to families.
In Proposition~\ref{proposition: key proposition}, the assumption that $\phi = \id$ at $x_0$ was imposed in order to ensure that $m_{x_0} = 0$.
However, when extending the argument to families, this assumption must be removed.
When working with the Riesz topology, this seems difficult because the Riesz transform is highly nonlinear.

We begin by recalling the definition of the gap topology.
Let $H$ be a separable Hilbert space over $\C$.
Recall that $\unboundedselfadjointoperators(H)$ and $\unitaryoperators(H),\boundedselfadjointoperators(H) \subset \boundedoperators(H)$ denote the spaces of unbounded selfadjoint operators, unitary operators, bounded selfadjoint operators, and bounded operators on $H$ respectively.
We define the Cayley transform $\Cayleytransform \colon \unboundedselfadjointoperators(H) \to \unitaryoperators(H)$ via functional calculus associated with an injective continuous map
\begin{equation}
  \Cayleytransform \colon \R \to U(1) \setminus \{1\}, \lambda \to \frac{\lambda-i}{\lambda+i}.
\end{equation}
The gap topology on $\unboundedselfadjointoperators(H)$ is defined as the pullback of the norm topology on $\unitaryoperators(H) \subset \boundedoperators(H)$ via $\Cayleytransform$.
We note that $\Image \Cayleytransform = \{U \in \unitaryoperators(H) \mid \Kernel (U - \id) = \{0\} \}$.
Moreover, we observe that
\begin{equation}
  \begin{aligned}
  \Cayleytransform(A) - \Cayleytransform(B) = {}& \frac{A-i}{A+i} - \frac{B-i}{B+i} \\
  = {}& \left( 1 - \frac{2i}{A+i} \right) - \left( 1 - \frac{2i}{B+i} \right) = -2i \left( \frac{1}{A+i} - \frac{1}{B+i} \right)
  \end{aligned}
\end{equation}
for $A,B \in \unboundedselfadjointoperators(H)$.
Hence, the gap topology is the weakest topology such that the maps $A \mapsto (A \pm i)^{-1}$ are continuous.
We also remark that the gap topology is strictly weaker than the Riesz topology.

The following proposition should be regarded as the counterpart, for the gap topology, of inequality~\eqref{equation: inequality form operator-Lipschitz function theory}, which has played a key role in the proof of Riesz continuity and follows from the theory of operator-Lipschitz functions.
\begin{proposition}
  We have
  \begin{equation}
    \operatornorm{\Cayleytransform(A+K) - \Cayleytransform(B+L)} \le 2\operatornorm{K-L} + (1+\operatornorm{L})^2 \operatornorm{\Cayleytransform(A)-\Cayleytransform(B)}
  \end{equation}
  for $A,B \in \unboundedselfadjointoperators(H)$ and $K,L \in \boundedselfadjointoperators(H)$.
\end{proposition}
\begin{proof}
  We have
  \begin{equation}
    \begin{aligned}
    & \Cayleytransform(A+K) - \Cayleytransform(B+L) \\
    ={} & 2i \left( (A+K+i)^{-1} - (B+L+i)^{-1} \right) \\
    ={} & 2i \left( (A+K+i)^{-1} - (A+L+i)^{-1} + (A+L+i)^{-1} - (B+L+i)^{-1} \right) \\
    ={} & 2i \left( (A+K+i)^{-1} - (A+L+i)^{-1} \right) + 2i \left( (A+L+i)^{-1} - (B+L+i)^{-1} \right).
    \end{aligned}
  \end{equation}
  By the second resolvent identity, we have
  \begin{equation}
    \begin{aligned}
    2i \big( (A+K+i)^{-1} - (A+L+i)^{-1}\big) ={} & 2i (A+K+i)^{-1} \left( (A+L+i) - (A+K+i) \right)(A+L+i)^{-1} \\
    ={} & 2i (A+K+i)^{-1} (L-K) (A+L+i)^{-1}.
    \end{aligned}
  \end{equation}
  Hence, we have
  \begin{equation}
    \operatornorm{2i \big( (A+K+i)^{-1} - (A+L+i)^{-1} \big)} = 2\operatornorm{(A+K+i)^{-1} (L-K) (A+L+i)^{-1}} \le 2\operatornorm{L-K}.
  \end{equation}
  Noting that both $(A+L+i)^{-1}(A+i)$ and $(B+i)(B+L+i)^{-1}$ extend to bounded operators, we have
  \begin{equation}
    \begin{aligned}
    & 2i \big((A+L+i)^{-1} - (B+L+i)^{-1} \big) \\
    ={} & (A+L+i)^{-1}(A+i) \cdot 2i \big( (A+i)^{-1} - (B+i)^{-1} \big) \cdot (B+i)(B+L+i)^{-1} \\
    ={} & (A+L+i)^{-1}(A+i) \cdot \left( \Cayleytransform(B) - \Cayleytransform(A) \right) \cdot (B+i)(B+L+i)^{-1}.
    \end{aligned}
  \end{equation}
  Thus, we obtain
  \begin{equation}
    \begin{aligned}
    & \operatornorm{2i ((A+L+i)^{-1} - (B+L+i)^{-1})} \\
    ={} & \operatornorm{(A+L+i)^{-1}(A+i)} \operatornorm{\Cayleytransform(B) - \Cayleytransform(A)} \operatornorm{(B+i)(B+L+i)^{-1}} \\
    \le{} & (1+\operatornorm{L}) \cdot \operatornorm{\Cayleytransform(B) - \Cayleytransform(A)} \cdot (1+\operatornorm{L}) \\
    ={} & (1+\operatornorm{L})^2 \operatornorm{\Cayleytransform(B) - \Cayleytransform(A)} .
    \end{aligned}
  \end{equation}
  This completes the proof.
\end{proof}

\section{Definition of spectral flow}\label{sec:sf}
  In our main theorems in Section~\ref{sec:main-theorem}, $K^1(I,\partial I)$
of an interval $I$ and its two endpoints $\partial I$, 
plays the essential role.
The elements of $K^1(I,\partial I)$ are classified by the spectral flow defined below.
In Section~\ref{sec:mod-2APS}, we also discuss the mod-two version of the
APS index of Dirac operators when it is real and skewsymmetric.
The corresponding mod-two spectral flow of lattice Dirac operators,
which is real, $\mathbb{Z}_2$-graded and selfadjoint, 
classifies elements of $KO^0(I, \partial I)$.
By abuse of notation, we also call an element of $K^1(I,\partial I)$  itself the spectral flow,
as well as an element of $KO^0(I,\partial I)$ the mod-two spectral flow.

Let $I=[-1,1]$ be an interval. Let $\mathcal{H} \to I$ 
be a Hilbert bundle over $I$ with fibre a separable Hilbert space over $\C$.
Suppose $\{h_t\}_{t\in I}$ 
is a Riesz-continuous one-parameter family of unbounded selfadjoint Fredholm  operators on $\mathcal{H}$
with $\mathrm{Ker}\;h_t =\{0\}$ for $t \in \partial I=\{-1,1\}$.
For each $t$, there exists $\Lambda_t$ such that $h_t$ has
  no spectrum in the range $\{ \lambda\;|\; -\Lambda_t \leq \lambda \leq \Lambda_t \}$
  except for finite eigenvalues with finite multiplicities.
Let us introduce a finite number of points 
  $t_0=-1 < t_1 < \cdots t_n=+1$ in $I$ for which
  we assign the values $\lambda_1, \ldots,\lambda_n$ 
  such that
\begin{enumerate}
\item
$\lambda_1=\lambda_n=0$.
\item
For any $t \in I$ in the range $t_{k-1} \leq t \leq t_k$, 
$\lambda_k$ is in the range $-\Lambda_t<\lambda_k<\Lambda_t$ and
is not an eigenvalue of $h_t$.
\end{enumerate}
For the $k$-th set $(t_k,\lambda_k)$ for $0<k<n$,
  we assign $\mathrm{sgn}_k$ and $d_k$ as follows.
  For $\lambda_k \neq \lambda_{k+1}$, we set
  \[
  \mathrm{sgn}_k = \frac{\lambda_k - \lambda_{k+1}}{|\lambda_k - \lambda_{k+1}|},
  \]
  and $d_k$ by sum of dimensions of the eigenspace
  with the eigenvalues in the range between $\lambda_k$ and $\lambda_{k+1}$.
  When $\lambda_k=\lambda_{k+1}$, we assign that $\mathrm{sgn}_k=0$ and $d_k=0$.
\begin{definition}[Spectral flow]
  Let $\HH$ be a complex Hilbert bundle over $I$
  and $\{h_t\}_{t\in I}$ be a Riesz-continuous family of unbounded selfadjoint Fredholm operators on $\HH$.
  The spectral flow of $\{h_t\}_{t\in I}$ is defined by
  $$\mathrm{sf}[\{h_t\}_{t\in I}]=\sum_{0<k<n} \mathrm{sgn}_k d_k.$$
It is known that (as explicitly shown in \cite{Aoki:2024sjc})
the definition is independent of the choice of the sets $\{t_k\}$ and $\{\lambda_k\}$,
and that $\mathrm{sf}[\{h_t\}_{t\in I}]$ can be identified as 
an element of $K^1(I,\partial I)$: 
there is an isomorphism from the element $[\{h_t\}_{t\in I}] \in K^1(I,\partial I)$ to $\mathrm{sf}[\{h_t\}_{t\in I}] \cong \mathbb{Z}$.

\end{definition}



When the Dirac operator is real and skewsymmetric,
every eigenvalue $\lambda$ makes a positive and negative pair: $\pm \lambda$.
Therefore, the APS index makes sense only modulo two.
For the corresponding spectral flow of the massive Dirac operators,
the standard spectral flow defined above is always zero and we need
a different way to evaluate the zero crossings of the paired spectrum modulo two.
The explicit definition of the mod-two spectral flow,
which classifies elements of $KO^0(I,\partial I)$,
is given in \cite{Carey2016SpectralFF, Fukaya:2020tjk}.
Here we just present its finite-dimensional version, since it is 
essential in extracting the mod-two APS index of lattice Dirac operators.

We consider a finite-dimensional real and  $\mathbb{Z}_2$-graded Hilbert space denoted by  $\hat{\mathcal{H}}_{\mathbb{R}}$.
Let $\{\hat{h}_t\}_{t \in I}$ be a smooth one-parameter family of real,
$\mathbb{Z}_2$-graded and selfadjoint operators
on $\hat{\mathcal{H}}_{\mathbb{R}}$.
We assume that $\hat{h}_{-1}$ and $\hat{h}_1$ are invertible.
Choose a real operator $V$ on $\hat{\mathcal{H}}_{\mathbb{R}}$ such that
\begin{equation}
 \hat{h}_1 = V^* \hat{h}_{-1} V.
\end{equation}
The mod-two spectral flow is defined by
\begin{equation}
  \mathrm{sf}_2[\{\hat{h}_t\}_{t \in I}] := \frac{1- \mathrm{sgn} \det V}{2}
  \; \in \mathbb{Z}_2.
\end{equation}

\section{Main theorem}
\label{sec:main-theorem}

Let $X := \T^d = (\R/\Z)^d$ be a $d$-dimensional flat torus and 
$\lX_a :=\lT$ be its lattice approximation with the lattice spacing $a=1/N$. 
We use the same setup as Section~\ref{sec:Dirac-operators} for the continuum and lattice Dirac operators
and the interpolating function  $\interpolator{a} : L^2(\lX_a, \lE)\to L^2(X,E)$ defined by ~(\ref{eq:from-lat-to-cont}). 
Let $Y \subset X$ be a separating submanifold that decomposes $X$ into the union of two compact manifolds
$X_+$ and $X_-$ which share the common boundary $Y$.

\subsection{Domain-wall mass term}

We define an $L^\infty$-function  $\kappa: X\to [-1,1]$ by $\kappa \equiv \pm 1$ on  $X_{\pm} \setminus Y$,
and its lattice version $\lkappa$ by the restriction $\kappa|_{\lX_a}$.
We also define a one-parameter family $\{\kappa_t : X\to [-1,1]\}_{t\in [-1,1]}$ as well as its lattice version by
\begin{equation}\label{eq:kappa_t}
 \kappa_t := \frac{1+t}{2}\kappa - \frac{1-t}{2}\id,\;\;\; \lkappa_t := \kappa_t|_{\lX_a}.
\end{equation}

\begin{proposition}\label{prop:DWproperty}
 There exists a constant $C$ such that
\begin{equation}
 || \kappa_t \interpolator{a}\phi - \interpolator{a}(\lkappa_t \phi)||_{L^2} \le C a^{1/d}||\phi||_{L_1^2}
\end{equation}
holds  for any  $\phi \in L^2(\lX_a, \lE_a)$.
\end{proposition}

Before proving Proposition \ref{prop:DWproperty}, let us introduce
another function $\bar{\kappa}_t: X\to [0,\infty)$ by
\begin{equation}
 \bar{\kappa}_t (x):= \sup\{|\lkappa_t(z)-\kappa_t(x)|\;| z\in \lX \mbox{ and } \rho_a(x-z)\neq 0\}. 
\end{equation}
Since $\lkappa_t(z)$ and $\kappa_t(x)$ are different only in the neighbourhood of the
domain-wall where $\bar{\kappa}_t$ takes almost a constant value,  the following propositions hold.
\begin{proposition}\label{prop:barkappa}
 For a non-negative integer $p$, there exists a constant $C_2$ such that
 \begin{equation}
  ||\bar{\kappa}_t||_{L^p} \le C_2 a^{1/p}
 \end{equation}
holds.
\end{proposition}
\begin{proof}
There exists $l>0$ such that for any $y \in Y$ and $x \in X$
if $\rho_a(x-y)\neq 0$ then $|x-y|< la$ for any $a$. 
Let $N_{la}(Y)$  be the $(la\mbox{-})$neighbourhood of the domain-wall $Y$
given by $N_{la}(Y) =\{x\in X\; |\; \exists y \in Y,\; |x-y|< la \}$.
Noting that the maximal value of $ \bar{\kappa}_t$ is two, 
and the volume of $N_{la}(Y)$ is $o(a)$,
we have an inequality
\begin{equation}
 \int_X | \bar{\kappa}_t|^p dx \le  C_2'la.
\end{equation}
with a $p$ dependent constant $C_2'$. Then the next proposition follows.
\end{proof}

\begin{proposition}\label{prop:barkappa2}
 For $\phi \in L^2(\lX_a, \lE_a)$, the following inequality holds. 
\begin{equation}
 || \kappa_t \interpolator{a}\phi - \interpolator{a}(\lkappa_t \phi)||_{L^2} \le ||\bar{\kappa}_t \interpolator{a}(|\phi|)||_{L^2}.
\end{equation}
\end{proposition}
\begin{proof}
 For $x\in X$, the explicit computation shows
\begin{equation}
 (\kappa_t \interpolator{a}\phi - \interpolator{a}(\lkappa_t \phi))(x) = \sum_{z\in \lX} \rho_a(x-z)(\kappa_t(x)-\lkappa_t(z))
\paralleltransport_{x,z}\phi(z).
\end{equation}
Therefore, we have the following inequalities
\begin{align}
 |\kappa_t \interpolator{a}\phi - \interpolator{a}(\lkappa_t \phi)|(x) & \le
 \sum_{z\in \lX} \rho_a(x-z)|\kappa_t(x)-\lkappa_t(z)|
|\paralleltransport_{x,z}\phi(z)|\\
& \le 
\bar{\kappa}_t(x) \sum_{z\in \lX} \rho_a(x-z)|\phi(z)| = \bar{\kappa}_t\interpolator{a}(|\phi|)(x),
\end{align}
which leads to the proposition above.
\end{proof}

\begin{proposition}\label{prop:barkappa3}
There exists a constant $C_3$ such that
 \begin{equation}
  ||\bar{\kappa}_t \interpolator{a}(|\phi|)||_{L^2} \le C_3 ||\bar{\kappa}_t||_{L^d}\cdot ||\interpolator{a}(|\phi|)||_{L_1^2}
 \end{equation}
holds for any $\phi \in L^2(\lX_a, \lE_a)$, where $d$ is the dimension of $X$.
\end{proposition}
\begin{proof}
 According to the H\"older inequality and Sobolev inequality
there exists a constant $C_3$ such that
\begin{align}
 ||\bar{\kappa}_t \interpolator{a}(|\phi|) ||_{L^2} \le ||\bar{\kappa}_t||_{L^d}\cdot ||\interpolator{a}(|\phi|)||_{L^p} 
\le ||\bar{\kappa}_t||_{L^d}\cdot C_3 ||\interpolator{a}(|\phi|)||_{L^2_1},
\end{align}
holds where $p$ satisfies $1/2=1/d+1/p$.
\end{proof}

From the above Propositions.~\ref{prop:barkappa},\ref{prop:barkappa2} and \ref{prop:barkappa3},
we have
\begin{align}
 || \kappa_t \interpolator{a}\phi - \interpolator{a}(\lkappa_t \phi)||_{L^2} 
 &\le ||\bar{\kappa}_t \interpolator{a}(|\phi|)||_{L^2}\\
 &\le C_3 ||\bar{\kappa}_t||_{L^d}\cdot ||\interpolator{a}(|\phi|)||_{L_1^2}\\
 &\le C_2 C_3 a^{1/d} ||\interpolator{a}(|\phi|)||_{L_1^2}\\
 &\le C_2 C_3 a^{1/d} C_1 ||\phi||_{L_1^2},
\end{align}
with a positive constant $C_1$, which proves the Proposition~\ref{prop:DWproperty}.


\subsection{Main theorem}
Let $I=[-1,1]$ be a line segment parametrized by $t$
whose two end points are denoted by $\partial I=\{-1,1\}$.
We compare two one-parameter families 
of the continuum and lattice domain-wall fermion Dirac operators:
\begin{equation}
 \{D-m\kappa_t \gamma\}_{t\in I},\;\;\; \{\WilsonDiracoperator-m\lkappa_t\gamma \}_{t\in I}, 
\end{equation}
with $\kappa_t$ and $\lkappa_t$ defined in the previous subsection.
We assume that at $t=1$ $D-m\kappa_1 \gamma$ is invertible.
$D-m\kappa_{-1} \gamma$ is also invertible which trivially follows
from $(D-m\kappa_{-1}\gamma)^2=D^2+m^2$.
From a general argument, we can show that $D-m\kappa_t \gamma$ is a Fredholm operator at any value of $t$.

According to Proposition~\ref{theorem: minus-variant2}, 
the unbounded continuum operator $D-m\kappa_t \gamma$ is Riesz-continuous 
with respect to $t$, since they at different $t$ differ only by the mass term $m\kappa_t \gamma$,
which is a compact bounded operator.
Therefore, the family $\{D-m\kappa_t \gamma\}_{t\in I}$ 
can be regarded as an element of $K^1(I,\partial I)$,
which is classified by the spectral flow
\begin{equation}
 \mathrm{sf}[D-m\kappa_t\gamma] 
\in \mathbb{Z}.
\end{equation}
Then the question is if the lattice Dirac operator family
 $\{\WilsonDiracoperator-m\lkappa_t\gamma \}_{t\in I}$
can be identified as a well-defined element of $K^1(I,\partial I)$ or not.

Let us define a continuum-lattice combined domain-wall fermion Dirac operator 
$D_a^{\rm cmb}(m,t,s): \Gamma(E)\oplus \Gamma(\lE_a)\to  \Gamma(E)\oplus \Gamma(\lE_a)$ by
\begin{align}
 D_a^{\rm cmb}(m,t,s) &:=\left(
    \begin{array}{cc}
       D-m\kappa_t\gamma  & s\interpolator{a}\\
       s\interpolator{a}^*   &   -(\WilsonDiracoperator-m\lkappa_t\gamma)
    \end{array}
  \right)
\\
&=
\left(
    \begin{array}{cc}
       D  & 0\\
       0  &   -\WilsonDiracoperator
    \end{array}
  \right)
-
m 
\left(
    \begin{array}{cc}
        \kappa_t\gamma & 0\\
      0  &   -\lkappa_t\gamma
    \end{array}
  \right)
+s
\left(
    \begin{array}{cc}
        0 & \interpolator{a}\\
       \interpolator{a}^*   &  0
    \end{array}
  \right)
\end{align}
with an additional parameter $s\in [0,1]$.

\begin{theorem}\label{goal}
Fix a staple-shaped parameter region $P$ in the $t$-$s$ plane
drawn in Figure~\ref{fig:st} starting from $(s,t)=(-1,0)$, via $(-1,1)$, $(1,1)$ then to $(1,0)$.
There exists a constant $a_1$ and $m_1$ such that for arbitrary 
lattice spacing $a=1/N$ satisfying $0< a \le a_1$, and arbitrary mass parameter $m>m_1$
$D_a^{\rm cmb}(m,t,s)$ is invertible at any point $(s,t)$ on $P$.
\end{theorem}

\begin{figure*}[tbh]
  \centering
  \includegraphics[width=10cm]{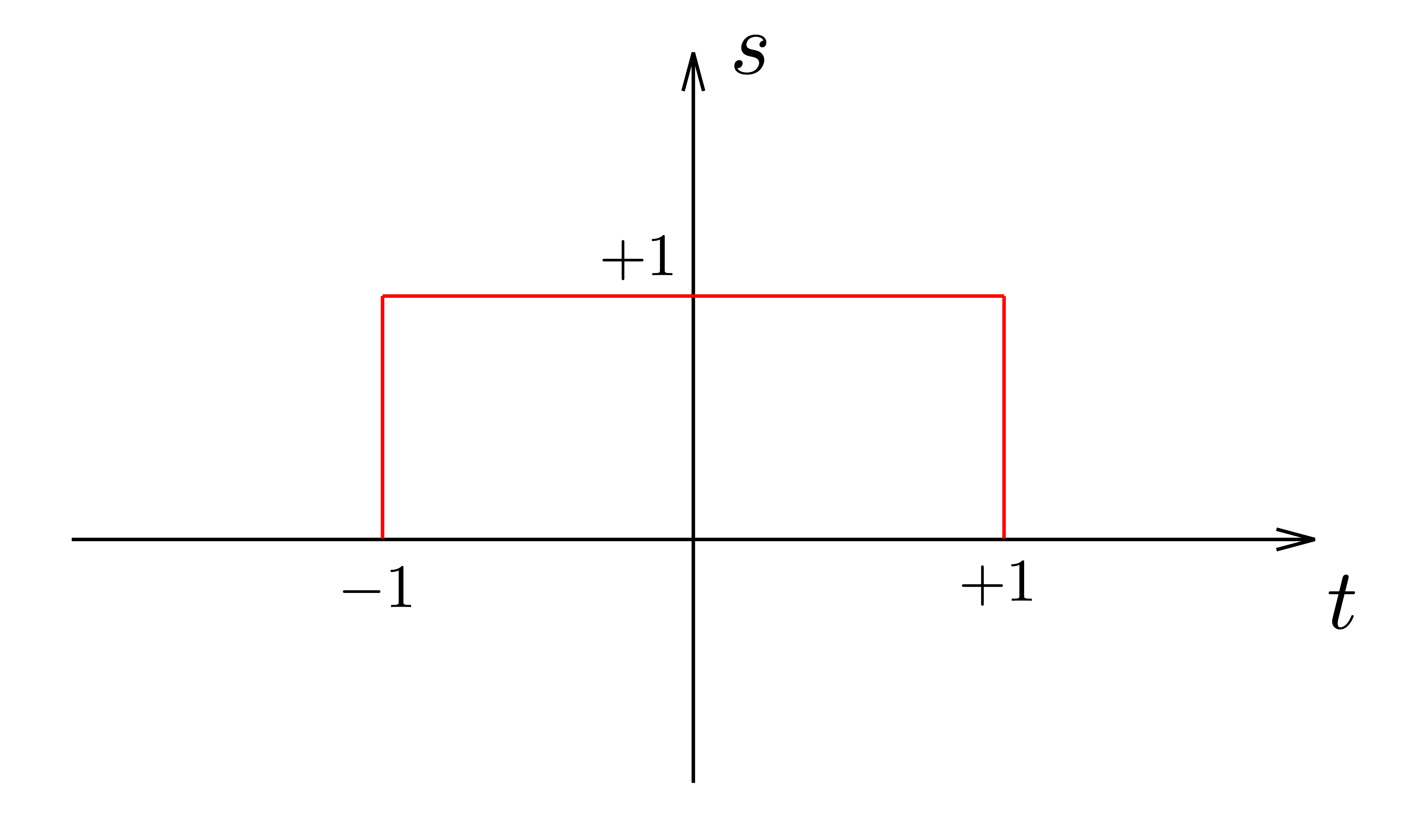}
  \caption{
 The staple-shaped parameter region in the $t$-$s$ plane
 where we prove that the lattice-continuum combined 
Dirac operator $D^{\rm cmb}(m,t,s)$ is invertible.
  }
  \label{fig:st}
\end{figure*}

Then the corollary of Theorem~\ref{goal} below and our main theorem follow.
\begin{corollary}
The spectral flow 
$\mathrm{sf}[\WilsonDiracoperator-m\lkappa_t\gamma]$ is
well-defined as an element of $K^1(I,\partial I)$,
where $I=[-1,1]$ and $\partial I=\{-1,1\}$.
\end{corollary}


\begin{theorem}[APS Index of lattice Dirac operator]\label{goalgoal}
For any lattice spacing $a=1/N<a_1$, the following holds.
\[
\mathrm{sf}[D-m\kappa_t\gamma] = \mathrm{sf}[\WilsonDiracoperator-m\lkappa_t\gamma]\in K^1(I,\partial I)\cong \mathbb{Z},
\]
where $I=[-1,1]$ and $\partial I=\{-1,1\}$.
\end{theorem}
\begin{proof}
The claim immediately follows from Proposition~\ref{theorem: minus-variant2}.
\end{proof}

As discussed in Section~\ref{eq:FFMOYY}, 
the spectral flow $\mathrm{sf}[D-m\kappa_t\gamma]$ equals to the 
APS index $ \APSindex (\restrictedTo{D}{X_+})$.
Therefore, we regard $\mathrm{sf}[\WilsonDiracoperator-m\lkappa_t\gamma]$
as a lattice formulation of the APS index.

Since $\WilsonDiracoperator-m\lkappa_t\gamma$ is a finite-sized matrix,
the following equality immediately follows 
by using the fact that $\eta(\WilsonDiracoperator+m\gamma)=0$ proved in Proposition~C.1 of \cite{Aoki:2024sjc}. 
\begin{corollary}
 \begin{equation}
  \mathrm{sf}[\WilsonDiracoperator-m\lkappa_t\gamma] = -\frac{1}{2}\eta(\WilsonDiracoperator-m\lkappa\gamma). 
 \end{equation}
\end{corollary}

\subsection{Proof of Theorem \ref{goal}}

Suppose that Theorem~\ref{goal}  does not hold.
Then there should exist a series labeled by $i=1,2,\ldots$ composed by
\begin{itemize}
 \item $a_i=1/N_i \to 0$,
 \item $(t_i, s_i) \in P$ 
(the staple-shaped region in Figure~\ref{fig:st}),
 \item $(\psi_{i}, \phi_{i}) \in L^2(E) \oplus L^2(\lE_{a_i})$
\end{itemize}
which satisfy for any $i$ that
\begin{enumerate}
\item
\begin{equation}
 || \psi_{i}||_{L^2}^2 + || \phi_{i}||_{L^2}^2 =1 
\end{equation}
and
\item
\begin{equation}
D_{a_i}^{\rm cmb}(m,t_i,s_i)
\left(
    \begin{array}{c}
      \psi_{i} \\
     \phi_{i} 
    \end{array}
  \right)
=0.\label{eq:contradiction}
\end{equation}
\end{enumerate}

Taking subsequences, we can assume without loss of generality
 that $(t_i,s_i)$ converges to  a point $(t_\infty,s_\infty)\neq (0,0)$:
$t_i \to t_\infty$ and $s_i \to s_\infty$.

Let us decompose ~\eqref{eq:contradiction} into the two equations.
\begin{align}
 (D-m\kappa_{t_i}\gamma) \psi_i  +s_i \interpolator{a_i}\phi_{i} &=0,\\
 s_i  \interpolator{a_i}^* \psi_i -(\WilsonDiracoperator_{a_i}-m\lkappa_{t_i}\gamma)\phi_{i}&=0.
\end{align}

From the first equation we have the following uniform bound,
\begin{align}
  ||\psi_i||^2_{L_1^2} &\le C(||D \psi_i||_{L^2}^2 + ||\psi_i||^2_{L^2}) 
= C(||m\kappa_{t_i}\gamma \psi_i - s_i \interpolator{a_i}\phi_{i} ||^2_{L^2} + ||\psi_i||^2_{L^2})
\\&
\le  C((m^2||\kappa_{t_i}||^2+1) ||\psi_i||^2_{L^2} + s_i^2 ||\interpolator{a_i}\phi_{i} ||^2_{L^2}) \le C'
\end{align}
with positive constants $C$ and $C'$.
For the last inequality we have used the property (\ref{property:iota*iota}) in Section~\ref{sec:properties}.
From the second equation, we have 
\begin{align}
 ||\phi_{i}||^2_{L^2_1}& \le C(||\WilsonDiracoperator_{a_i}\phi_{i}||^2_{L^2}+ ||\phi_{i}||^2_{L^2})
\\&\le
 C(s_i^2 ||\interpolator{a_i}^* 
\psi_i
||^2_{L^2} + (m^2||\lkappa_{a_i}||^2+1) ||\phi_i||^2_{L^2}) <C',
\end{align}
with positive constants $C$ and $C'$. Here we have used the property (\ref{property:iotaiota*}) in Section~\ref{sec:properties}.
Thus, $\psi_i$ and $\interpolator{a_i}\phi_{i}$ are uniformly $L^2_1$ 
bounded
in $L^2(E)$.

Taking subsequences, we can assume without loss of generality
that $\interpolator{a_i}\phi_{i} $ weakly converges in $L^2_1(E)$ to a vector $\psi'_\infty$ 
and $\psi_{i}$ weakly converges in $L^2_1(E)$ to another vector $\psi_\infty$.
Moreover, these are strong convergence in $L^2$ according to the Rellich theorem.
From the property (\ref{property:iota*iota}) in Section~\ref{sec:properties}, 
we can also show that $||\interpolator{a_i}\phi_{i}||^2_{L^2}-||\phi_{i}||^2_{L^2}\le Ca_i||\phi_{i}||^2_{L_1^2}$
with a constant $C$
so that the series $\phi_{i}$ strongly converges to a lattice vector $\phi_\infty$ in $L_2$.
Then we can conclude that $||\psi_\infty||_{L^2}^2+||\phi_\infty||_{L^2}^2=1$.

\begin{lemma}\label{weak convergence of equation}
With the assumptions made above for the proof by contradiction,
the following equation holds.
\begin{equation*}
 \left(
    \begin{array}{cc}
        D-m\kappa_{t_\infty}\gamma & s_\infty \\
       s_\infty   &  -(D-m\kappa_{t_\infty}\gamma)
    \end{array}
  \right)
\left(
    \begin{array}{c}
       \psi_\infty \\
       \psi'_\infty 
    \end{array}
  \right)
=0.
\end{equation*}
\end{lemma}
\begin{proof}
For any $\Psi \in C^\infty(E)$, 
the weak limit of the inner product $\langle D\psi_i, \Psi\rangle_{L^2}$ becomes
\begin{equation}
 \langle D\psi_i, \Psi\rangle_{L^2} = \langle \psi_i, D^*\Psi\rangle_{L^2}
\to \langle \psi_\infty, D^*\Psi\rangle_{L^2} = \langle D\psi_\infty, \Psi\rangle_{L^2}.
\end{equation}
Similarly, we have
\begin{align}
 \langle m\kappa_{t_i}\gamma \psi_i-s_i(\interpolator{a_i}\phi_{i}), \Psi \rangle_{L^2}
&\to 
\langle m\kappa_{t_\infty}\gamma \psi_\infty-s_\infty\psi'_\infty, \Psi \rangle_{L^2}.
\end{align}
From the upper component of ~\eqref{eq:contradiction}, we have
\begin{equation}
 \langle (D-m\kappa_{t_\infty}\gamma) \psi_\infty  +s_\infty  \psi'_\infty, \Psi\rangle_{L^2} =0.
\end{equation}

By applying $\interpolator{a_i}$ to the lower component of ~\eqref{eq:contradiction}, we obtain an equation
in $L^2(E)$,
\begin{equation}
  \label{eq:latconv}
 -\interpolator{a_i}(\WilsonDiracoperator_{a_i}-m\lkappa_{t_i}\gamma) \phi_{i}  + s_i  \interpolator{a_i}\interpolator{a_i}^* \psi_{i} =0. 
\end{equation}
For $\Psi \in C^\infty(E)$, we have
\begin{equation}
 \langle \interpolator{a_i}(\WilsonDiracoperator_{a_i} \phi_{i}),\Psi\rangle_{L^2} 
 = \langle  \phi_{i}, \WilsonDiracoperator^*_{a_i} (\interpolator{a_i}^*\Psi)\rangle_{L^2}, 
\end{equation}
and 
\begin{align}
  \langle \interpolator{a_i} (\WilsonDiracoperator_{a_i}\interpolator{a_i}^*\interpolator{a_i}) \phi_{i}, \Psi\rangle_{L^2}
&= \langle \interpolator{a_i}^*\interpolator{a_i} \phi_{i}, \WilsonDiracoperator^*_{a_i} (\interpolator{a_i}^*\Psi)\rangle_{L^2}
\\&
= \langle \interpolator{a_i} \phi_{i},\interpolator{a_i} \WilsonDiracoperator^*_{a_i} (\interpolator{a_i}^*\Psi)\rangle_{L^2}
\to
\langle \psi'_\infty, D \Psi\rangle_{L^2},
\end{align}
where we have used the property (\ref{property:Dirac}) in Section~\ref{sec:properties}.
Comparing the above two, we obtain
\begin{align}
 |\langle \interpolator{a_i}(\WilsonDiracoperator_{a_i} \phi_{i}),\Psi\rangle_{L^2} -\langle \interpolator{a_i} (\WilsonDiracoperator_{a_i}\interpolator{a_i}^*\interpolator{a_i}) \phi_{i}, \Psi\rangle_{L^2}|\\
\le ||(\interpolator{a_i}^*\interpolator{a_i}-\id)\phi_{i}||_{L^2}\cdot ||\WilsonDiracoperator_{a_i}^*(\interpolator{a_i}^* \Psi)||_{L^2} \to 0,
\end{align}
where we have used the property (\ref{property:iota*iota}) in Section~\ref{sec:properties}
and that $||\WilsonDiracoperator_{a_i}^*(\interpolator{a_i}^* \Psi)||_{L^2}$ is bounded by 
$C||\Psi||_{L_1^2}$ with some constant $C$.
Therefore, we have for any $\Psi \in C^\infty(E)$, 
\begin{equation}
 \langle \interpolator{a_i}(\WilsonDiracoperator_{a_i} \phi_{i}),\Psi\rangle_{L^2} \to \langle D\psi_\infty', \Psi\rangle_{L^2}.
\end{equation}
Similarly, for any $\Psi \in C^\infty(E)$, the following weak convergence is obtained.
\begin{equation}
 \langle s_i \interpolator{a_i}\interpolator{a_i}^* \psi_i, \Psi\rangle_{L^2} =
 \langle  \psi_i, s_i \interpolator{a_i}\interpolator{a_i}^*\Psi\rangle_{L^2} \to \langle s_\infty \psi_\infty,  \Psi\rangle_{L^2}.
\end{equation}
Finally from the bound below for any $\Psi \in C^\infty(E)$, we have
\begin{align} 
 |\langle m \interpolator{a_i} (\lkappa_{t_i}\gamma \phi_i),\Psi \rangle_{L^2} - \langle m \kappa_{t_i} \gamma \interpolator{a_i}\phi_i,\Psi) \rangle | &\le m ||(\interpolator{a_i}\lkappa_{t_i} -  \kappa_{t_i} \interpolator{a_i})\gamma\phi_i||_{L^2}\cdot ||\Psi||_{L^2}
\\&\le m Ca_i^{1/d}|| \gamma\phi_i ||_{L^2}\cdot ||\Psi||_{L^2},
\end{align}
where Proposition~\ref{prop:DWproperty} is used in the second inequality, and
we have the following weak convergence of the series
\begin{align}
 \langle m \interpolator{a_i} (\lkappa_{t_i}\gamma \phi_i),\Psi \rangle_{L^2}  
\to \langle m \kappa_{t_\infty} \gamma \psi'_\infty, \Psi\rangle_{L^2}  
\end{align}
and 
\begin{equation}
 \langle s_\infty \psi_\infty -(D-m\kappa_{t_\infty}\gamma) \psi'_\infty, \Psi \rangle_{L^2}=0,
\end{equation}
holds.
\end{proof}
When Lemma~\ref{weak convergence of equation} holds, 
by applying the square of the operator, we have
\begin{align}
 & \left(
    \begin{array}{cc}
        D-m\kappa_{t_\infty}\gamma & s_\infty \\
       s_\infty   &  -(D-m\kappa_{t_\infty}\gamma)
    \end{array}
  \right)^2 \left(
    \begin{array}{c}
       \psi_\infty \\
       \psi'_\infty 
    \end{array}
  \right)  \\ &=
\left(
    \begin{array}{cc}
        (D-m\kappa_{t_\infty}\gamma)^2 + s_\infty^2 & 0 \\
        0   &  (D-m\kappa_{t_\infty}\gamma)^2+s_\infty^2
    \end{array}
  \right) \left(
    \begin{array}{c}
       \psi_\infty \\
       \psi'_\infty 
    \end{array}
  \right) \\ &= 0.
\end{align} 
Since $(D-m\kappa_{\pm 1}\gamma)^2$ is invertible
from the assumption, $(D-m\kappa_{t_\infty}\gamma)^2 + s_\infty^2$
is everywhere invertible on the staple-shaped region $(t_\infty, s_\infty) \in P$.
Therefore, $\psi_\infty$ and $\psi'_\infty$ must be zero,
which contradicts with the condition $||\psi_\infty||_{L^2}^2+||\phi_\infty||_{L^2}^2=1$. 
This proves Theorem~\ref{goal}.
Theorem~\ref{goalgoal} immediately follows from Theorem~\ref{goal}.
 \qed

\section{Applications to the mod-two APS index} 
\label{sec:mod-2APS}

So far, we have discussed the standard $\mathbb{Z}$-valued index
of Dirac operators in even dimensions and
the corresponding spectral flow of the massive Dirac operators,
which 
give elements of $K^1(I,\partial I)$.
Since our formulation of 
$K^n$ or $KO^n$ 
for arbitrary degree $n$ 
in Section~\ref{sec:K} is so general that
we can deal with unbounded continuum Dirac operators and bounded lattice Dirac operators
with the corresponding symmetries.
In particular, it is important in physics to formulate
the $\mathbb{Z}_2$-valued index of real skewsymmetric
Dirac operators in order to describe non-local types of anomaly \cite{Witten:1982fp}.
In the original definition, the $\mathbb{Z}_2$-valued index
gives an element of the group $KO^{-1}(\mathrm{pt})$,
where the degree $n=-1$ indicates absence of the $\mathbb{Z}_2$-grading or chirality operator.
The corresponding mod-two version of the 
spectral flow 
gives an element of the group $KO^0(I,\partial I)$
\footnote{
  A similar isomorphism $KO^{-2}(\mathrm{pt})\cong KO^{-1}(I,\partial I)\cong \mathbb{Z}_2$ also works.
}.
In this section, we formulate a lattice version
of the mod-two APS index using the mod-two spectral flow.

Let $X := \T^d = (\R/\Z)^d$ be a $d$-dimensional flat torus,
where $d$ can be any positive integer.
In this section, it is convenient to introduce the basis 
$\epsilon_1, \dots, \epsilon_d$
of the Clifford algebra  $\cliffordalgebra_d$, 
satisfying the anticommutation relation
\begin{equation}
  \{\epsilon_i, \epsilon_j\}=\epsilon_i \epsilon_j + \epsilon_j \epsilon_i = 2 \delta_{ij}.
\end{equation}
Let $E\to X$ be a Clifford module bundle on $X$ with 
a smooth bundle map $\cliffordmultiplication \colon X \times \cliffordalgebra_d \to \End(E)$.
We assume that $E$ is of the form $E^\mathbb{R}\otimes_\mathbb{R} \mathbb{C}$ for an Euclidean vector bundle $E^\mathbb{R}$ and $\sigma(\epsilon_i )$ is a real symmetric operator on $E$, {\it i.e.}, 
a symmetric operator on $E^\mathbb{R}$  for all $i$.


Let $Y \subset X$ be a separating submanifold that decomposes $X$ into the union of two compact manifolds $X_+$ and $X_-$ which share the common boundary $Y$.
We introduce the same one-parameter family of the
domain-wall function $\kappa_t$ as in ~\eqref{eq:kappa_t},
where $t$ is in the range $I=[-1,1]$ and the two end points are denoted by $\partial I$.

Fix a Clifford connection $A$ on $E$ preserving $E^\mathbb{R}$,
which determines the covariant 
derivative $\nabla_j$ in the $\bm{e}_j$ direction for each $j=1,\cdots d$,
and the parallel transport as in Section~\ref{sec:Dirac-operators}.
We denote the space of smooth sections of $E$ by $\Gamma(E)$.
Let $D:\Gamma (E)\to \Gamma (E)$ be a real and skewsymmetric
first order elliptic operator defined by
\begin{equation}
  D u:= \sum_{j=1}^d \cliffordmultiplication(\epsilon_j) \nabla_{j} u,
\end{equation}
for $u \in \Gamma (E)$.

A distinct property of the fermion system in odd dimensions
is that the massive Dirac operators $D - m\kappa_t \id$,
with a positive parameter $m>0$,
is neither selfadjoint nor skewadjoint\footnote{Recall that in even dimensions, we can use the $\mathbb{Z}_2$-grading operator $\gamma$
to make $\gamma (D-m\kappa_t\id)$ selfadjoint.} in general.
In order to 
represent the $KO^0(I,\partial I)$ elements,
we define the following one-parameter family of the
real symmetric operators $\{H_t\}_{t\in I}$ on $\Gamma(E)\oplus \Gamma(E)$ where
\begin{equation}
  H_t : 
  = \begin{pmatrix}0 & D-m\kappa_t\id \\ D^*-m\kappa_t\id  & 0\end{pmatrix},
\end{equation}  
with an assumption that $H_{t=1}$ is invertible ($H_{t=-1}$ is also trivially invertible).
This operator is Riesz-continuous 
with respect to $t$.

When $X_-=\{\emptyset\}$ and  $\kappa_t=t\id$, it is not difficult to
confirm that the mod-two spectral flow: $\mathrm{sf}_2[\{H_t\}_{t\in I}]$, 
which is the number of
zero-crossing pairs of the eigenvalues of $H_t$ along $t\in [-1,1]$,
agrees with the dimension of
$\mathrm{Ker} D$ modulo two.
In \cite{Fukaya:2020tjk}, it was proved that
the mod-two spectral flow of $\{H_t\}_{t\in I}$ equals to
the mod-two APS index of $D$ on $X_+$.

Now let us define the lattice Wilson Dirac operator
$\WilsonDiracoperator \colon \Gamma(\lE) \to \Gamma(\lE)$
using the same notation as in Section~\ref{sec:Wilson},
\begin{equation}
  \WilsonDiracoperator :=
  \sum_{i=1}^{d} \cliffordmultiplication(\epsilon_j) \differenceoperator_j
  +\frac{a}{2} \sum_{j=1}^d \forwarddifferenceoperator_j \big(\forwarddifferenceoperator_j\big)^*,
\end{equation}
which is real.
Note here that the first term is skewadjoint but the
second Wilson term is selfadjoint.
We also define the lattice version of the real symmetric operator
$\lattice{H}_t : \Gamma(\lE_a) \oplus \Gamma(\lE_a) \to \Gamma(\lE_a) \oplus \Gamma(\lE_a)$ by
\begin{equation}
\lattice{H}_t :=  
\begin{pmatrix}0 & \WilsonDiracoperator-m\lkappa_t\id \\ \WilsonDiracoperator^*-m\lkappa_t\id  & 0\end{pmatrix}.
\end{equation}

Let us define a continuum-lattice combined domain-wall fermion Dirac operator 
$H_a^{\rm cmb}(m,t,s): \Gamma(E)\oplus \Gamma(E) \oplus \Gamma(\lE_a) \oplus \Gamma(\lE_a)\to  \Gamma(E)\oplus \Gamma(E) \oplus \Gamma(\lE_a) \oplus \Gamma(\lE_a)$ by
\begin{align}
 H_a^{\rm cmb}(m,t,s) &:=\left(
    \begin{array}{cc}
       H_t  & \id_{2\times 2} \otimes s\interpolator{a}\\
       \id_{2\times 2} \otimes s\interpolator{a}^*   &   -\lattice{H}_t
    \end{array}
  \right)
\end{align}
with an additional parameter $s\in [0,1]$.

\begin{theorem}\label{goal-mod2}
On the path $P$ in Figure~\ref{fig:st},
there exists a constant $a_1$ and $m_1$ such that for arbitrary 
lattice spacing $a=1/N$ satisfying $0< a \le a_1$, and arbitrary mass parameter $m>m_1$
$H_a^{\rm cmb}(m,t,s)$ is invertible at any point $(s,t)$ on $P$.
\end{theorem}

\begin{corollary}
The mod-two spectral flow 
$\mathrm{sf}_2[\{\lattice{H}_t\}_{t\in I}]$ is
well-defined as an element of $KO^0(I,\partial I)$.
\end{corollary}

\begin{theorem}[Mod-two APS Index of lattice Dirac operator]
  \label{goalgoal-mod2}
For any lattice spacing $a=1/N<a_1$, the following holds.
\[
\mathrm{sf}_2[\{H_t\}_{t\in I}] =
\mathrm{sf}_2[\{\lattice{H}_t\}_{t\in I}] \in KO^0(I,\partial I).
\]
\end{theorem}

\begin{proof}
The proof goes in the essentially same way as that of Theorems~\ref{goal} and \ref{goalgoal}.
\end{proof}

When $Y$ has a collar neighbourhood as in Section~\ref{eq:FFMOYY}, 
the spectral flow $\mathrm{sf}_2[\{H_t\}_{t\in I}]$ equals to the 
mod-two APS index $ \APSindex^\mathrm{mod\textrm{-}2} (\restrictedTo{D}{X_+})$,
which was proved in \cite{Fukaya:2020tjk}.
Therefore,  $\mathrm{sf}_2[\{\lattice{H}_t\}_{t\in I}]$
can be regarded as its lattice formulation.

\section*{Acknowledgments}

The authors thank Yoshio Kikukawa and Yosuke Kubota
for helpful discussions.
This work was supported in part by JSPS KAKENHI Grant
(JP21K03222, JP21K03574, JP23K03387, JP23K22490, JP24K06719, JP25K07283, JP25K1738).
The work of SA was supported in part by RIKEN Special Postdoctoral Researchers Program.

\bibliographystyle{alpha}
\bibliography{references_bib.bib}
\end{document}